\documentclass[12pt,amsfonts, epsfig]{amsart}

\usepackage{amsmath, amscd, amssymb}
\usepackage[frame,cmtip,arrow,matrix,line,graph,curve]{xy}
\usepackage{epsfig}
\usepackage{graphpap, color}
\usepackage[mathscr]{eucal}
\usepackage{mathrsfs}

\numberwithin{equation}{section}

\def\sO{{\mathscr O}}
\def\sM{{\mathscr M}}

\def\sL{{\mathscr L}}
\def\sI{{\mathscr I}}
\def\sO{\mathscr{O}}

\def\sI{\mathscr{I}}

\def\sV{\mathscr{V}}

\newcommand{\CC}{\mathbb{C}}

\newcommand{\PP}{\mathbb{P}}

\newcommand{\ZZ}{\mathbb{Z}}

\newcommand{\bk}{\mathbf{k}}

\newcommand{\kk}{\bk}

\newcommand{\cal}{\mathcal}

\def\cE{{\cal E}}

\def\cU{{\cal U}}
\def\cV{{\cal V}}

\def\cS{{\cal S}}
\def\cX{{\cal X}}
\def\cY{{\cal Y}}

\def\mapright#1{\,\smash{\mathop{\lra}\limits^{#1}}\,}

\def\sta{^\ast}

\def\upmo{^{-1}}
\def\sta{^{\ast}}

\def\sta{^*}

\def\lra{\longrightarrow}

\newcommand{\bfmS}{{\bf mS}}

\def\begeq{\begin{equation}}
\def\endeq{\end{equation}}
\def\and{\quad{\rm and}\quad}
\def\bl{\bigl(}
\def\br{\bigr)}

\def\sub{\subset}
\def\Ao{{\mathbb A}^{\!1}}

\def\and{\quad\text{and}\quad}
\def\mapright#1{\,\smash{\mathop{\lra}\limits^{#1}}\,}

 \DeclareMathOperator{\Ext}{Ext}
  \DeclareMathOperator{\Hom}{Hom}

\let\lab=\label

\newtheorem{prop}{Proposition}[section]

\newtheorem{lemm}[prop]{Lemma}

\newtheorem{rema}[prop]{Remark}

\newtheorem{defi}[prop]{Definition}

\theoremstyle{definition}
\newtheorem{say}[prop]{}

\def\bone{{\mathbf 1}}

\def\Po{{\mathbb P^1}}

\def\Pf{{\mathbb P}^4}
\def\Pn{{\mathbb P}^n}
\def\PP{{\mathbb P}}

\def\MPd{\overline M_1(\Pn,d)}

\def\MPdg{\overline{M}_g(\Pn,d)}

\def\wMPdg{\widetilde{M}_g(\PP,d)}

\def\sta{^\ast}

\let\lab=\label

\def\sO{{\mathscr O}}

\def\lab#1{\label{#1}[{#1}]\  }

\def\lab{\label} %{{\bf index}-}%=\label

\def\beq{\begin{equation}}
\def\eeq{\end{equation}}

\title[Genus one stable maps]
{An Invitation to the Local Structures of Moduli of Genus one
stable maps}

\author{Yi Hu}
\address{Department of Mathematics, University of Arizona, USA.}
\email{yhu@math.arizona.edu}


\begin{document}
\maketitle

\begin{abstract}
This informal note provides some elementary examples to motivate
the local structural results of \cite{HL08} on the
 moduli space of genus one stable maps to projective space. The
 hope is that these examples will be helpful for graduate students to learn
 this important subject.
\end{abstract}

\section{Introduction}

The moduli space $\overline{M}_g$ of stable curves of genus $g$
has been an important subject of study in algebraic geometry. It
is a smooth Deligne-Mumford stack (orbifold). The moduli space
$\MPdg$ of stable maps of degree $d$ from genus-$g$ curves to the
projective space $\Pn$ is a natural generalization of
$\overline{M}_g$, but this generalization leads us from a smooth
space to a space that can contain singularities as bad as possible
(Vakil's Murphy's Law). For the purposes of some programs, it is
important to obtain the (local) structures of the moduli space
$\MPdg$.

 A point $[u, C]$
of $\MPdg$ is (the isomorphism class of) a map
$$u: C \lra \Pn$$
where $C$ is an algebraic curve of arithmetic genus $g$ with at
worse nodal singularities  such that the automorphism group of the
map $[u, C]$ is finite.  Here a nodal singularity locally is given
by $(xy=0)$; the point corresponding to the origin is called a
node. An automorphism of $[u, C]$ means a morphism $\phi: C \lra
C$ such that $u \circ \phi = u.$  The automorphism group of the
map $[u, C]$ is finite if and only if  any genus-0 irreducible
component
 of $C$ contains at least three  nodes whenever it
is contracted by the map and any genus-1 irreducible component of
$C$ contains at least one node whenever it is contracted by the
map.

One may add $m$ marked points to the domain curve $C$ away from
the nodes and obtain the moduli space of  stable maps with $m$
markings, denoted  $\overline{M}_{g,m}(\Pn,d)$. But, as far as
singularity is concerned, $\overline{M}_{g,m}(\Pn,d)$ and $\MPdg$
have exactly the same local singularity types simply because each
marked point moves in a smooth local domain. Hence, as far as
singularity types are concerned, we may only consider moduli
spaces without markings.

When $g=0$, $\overline{M}_{0,m}(\Pn,d)$ is smooth (as a stack or
orbifold); when $g \ge 1$, $\MPdg$ is  singular (as a stack or
orbifold). Indeed, if allowing arbitrary $g$ and $d$, Ravi Vakil
showed that $\MPdg$ can contain all possible singularity types
over $\ZZ$. This seems to be a piece of bad news, but thinking
positively, it also makes the spaces $\MPdg$ ultimately rich as
modular singularity models to investigate.

Historically though, $\MPdg$ was introduced for the (sole) purpose
of defining the Gromov-Witten invariants. An important standing
problem in this area is to enumerate the (virtual) number of
curves with fixed genus and degree in a smooth Calabi-Yau
threefold in $\Pf$, or more generally, in a complete intersection
$X$ in $\Pn$. In principle, the curve-counting business on $X$ can
be done on the ambient space $\Pn$, using the defining equations
of $X$. The moduli space $\overline{M}_g (X,d)$ of stable maps of
 degree $d$ from curves of genus $g$ to $X$ is naturally a
 submoduli space of $\overline{M}_g (\Pn,d)$.
For example, let us assume that $X$ is a smooth hypersurface
defined by $X= s^{-1}(0)$ for some section $s \in \Gamma(\Pn,
\sO_{\Pn}(k))$ with some positive integer $k$. If we let
\begin{equation*}
\begin{CD}
\cX @>{f}>> \Pn \\
@V{\pi}VV \\
\MPdg
\end{CD}
\end{equation*}
be  the universal family over the moduli space $\MPdg$ with the
universal map $f$ and let
$$\sigma = \pi_*f^* s \in \Gamma(\MPdg,\pi_*f^*\sO_{\Pn}(k)),$$ then
we have $$\overline{M}_g (X,d)=\sigma^{-1}(0).$$ If
$\pi_*f^*\sO_{\Pn}(k)$ were locally free, then it would have a
natural Euler class, and this Euler class would  bridge the
intersection theory on $\overline{M}_g (X,d)$ to the intersection
theory on $\MPdg$. Thus, resolving the non-locally free locus of
the direct image sheaf  $\pi_*f^*\sO_{\Pn}(k)$ is essential in GW
theory of complete intersections in $\Pn$. Here, by resolving the
sheaf, we mean a diagram
\begin{equation*}
\begin{CD}
{\tilde f}: \widetilde{\cX} @>>> \cX @>{f}>> \Pn \\
@V{\tilde\pi}VV @V{\pi}VV \\
\wMPdg @>>> \MPdg,
\end{CD}
\end{equation*}
where $\wMPdg \lra \MPdg$ is a blowup and $\widetilde{\cX}= \cX
\times_{\MPdg} \wMPdg $
 such that the direct image sheaf $\tilde\pi_*{\tilde f}^*\sO_{\Pn}(k))$ is locally free.

Thus, one sees that on the one hand, it is important for GW theory
to resolve the sheaf $\pi_*f^*\sO_{\Pn}(k)$; on the other hand, it
is important for singularity theory to resolve $\MPdg$. The two
problems are related. The point of \cite{HL08} is that it is more
natural to resolve the sheaves $\pi_*f^*\sO_{\Pn}(k)$ first, and
at least when the genus is low, resolving the sheaves will also
ensure a resolution  of the moduli space.

To see this point, observe that  a stable map $[u]\in \MPdg$, as a
morphism, is given by the data
$$u=[u_0,\cdots,u_n]: C\lra\Pn,\quad u_i\in H^0(u\sta \sO_{\Pn}(1));
$$
its deformation is determined by the combined deformation of the
 curve $C$ and the sections $\{u_i\}$. Since the
deformation of the curve is unobstructed, the irregularity of
$\MPdg$ is closely related to the non-local-freeness of the direct
image sheaf $\pi_* f^* \sO_{\PP}(1)$   This alludes that
desingularizations of $\MPdg$ {\sl should be} governed by
desingularizations  of $\pi_* f^* \sO_{\PP}(1)$. For genus-1, this
is true in the simplest form, a desingularization of $\pi_* f^*
\sO_{\PP}(1)$ implies a desingularization of $\MPd$.

Vakil and Zinger first discovered a  desingularization of the main
component of the moduli space $\MPd$ in \cite{VZ}. Their method is
analytic in nature. They found a natural sequence of blowups that
resolve singularities of $\MPd$ and then showed that the same
blowups also resolve the sheaves $\pi_* f^* \sO_{\PP}(k)$.  As
hinted in the last paragraph, from the algebro-geometric approach,
it is more natural to resolve the sheaves first. In \cite{HL08},
we first obtain local structures of the sheaf $\pi_* f^*
\sO_{\PP}(k)$. The structures of $\pi_* f^* \sO_{\PP}(1)$ allow us
to derive local defining equations of the moduli space $\MPd$.
Having obtained these local equations, it is rather clear what
loci one should blow up, how the resulting space turns out to be
smooth, and why the resulting direct image sheaves become locally
free.

This note, through some  concrete examples, is solely devoted to
reveal the structures of the sheaf $\pi_* f^* \sO_{\Pn}(k)$. By
working out these examples, we hope the student will familiarize
himself/herself with the aspects of families of elliptic curves
that are  useful for obtaining the local structures of the moduli
space of  stable maps.

%{\sl Acknowledgements}.
\medskip
During the summer school, I gave four lectures on tropical curves
and their applications to plane enumerative geometry. When the
school was over, the organizers asked every speaker to write up
his lecture notes for the proceeding. However, as there have been
already several excellent expository articles on tropical curves
that the students can easily find on arXiv and I did not feel that
I could make any meaningful improvement,  instead, I thought that
an introductory note on elliptic stable maps should help students
to learn this important subject. The idea is that through some
examples the students will gain the intuition about the approach
to the local structures of the stable map muduli \cite{HL08}. I
hope that the material will be a more converging addition to the
proceeding, and it will be more useful for graduate students as
well as for researchers.  Most of the note should be accessible to
any graduate student with some backgrounds on algebraic geometry.

{\it This note recollects the toy examples that Jun Li and I
calculated in the summer of 07 as the warm-up as well as the guide
for our approach toward the general theory \cite{HL08}, but,
needless to say,  all the mistakes in this detailed presentation
must be due to my own oversight.}  I thank Jun Li, from whom I
have learned a lot, for the collaboration. I also thank CMS of
Zhejiang University and the organizers of the summer school,
especially Lizhen Ji, for the excellent environment and for
supporting the idea to include this note in the proceeding.

\tableofcontents

\section{The Structures of the Direct Image Sheaf}

We begin with motivating the setups of our examples.

\subsection{Motivation: reduction to local family}
\begin{say} Let  $\pi: \cX \lra \MPd$ be the universal family with
the universal map $f: \cX \lra \Pn$. As for any sheaf, the
question on the structures of the direct image sheaf $\pi_*f^*
\sO_{\Pn}(k)$ is (\'etale) local. Given any point
$$[u,C]=[u: C
\lra \Pn] \in \MPd,$$ we choose a small \'etale neighborhood $\cU
\ni [u,C]
 \subset \MPd$. By choosing $k$ general hyperplanes $H_1, \cdots, H_k$ of $\Pn$, we
can assume that $$f^*(H_1+\cdots+H_k) \cap C$$ is a simple divisor
$\sum_{i=1}^m s_i$ of degree $m=dk$ ($s_1, \cdots, s_m$ are
disjoint). Let $\cS=f^*(H_1+\cdots+H_k)$, then $$f^* \sO_{\Pn}
(k)= \sO_\cX (\cS).$$ By an \'etale base change, we may assume
that $\cS=\sum_{i=1}^m \cS_i$ where each $\cS_i$ is a section of
$\pi:\cX\to \cV$ such that $s_i = \cS_i \cap C$. Hence, the local
structures of $\pi_*f^* \sO_{\Pn}(k)$ is reflected in the
structure of $\pi_* \sO_\cX (\cS)$.
\end{say}

\begin{say} This motivates us to consider some examples of flat
families of elliptic curves, $\pi: Z \lra B$ with a section $S$,
and study the associated direct image sheaf $\pi_* \sO_Z(mS)$.
\end{say}

\begin{defi} The core $C_e$ of a connected genus-one curve $C$ is the unique
smallest (by inclusion) subcurve of arithmetic genus one.
\end{defi}

\vskip 4cm

\begin{picture}(3, 5)
\put(90,5){ \psfig{figure=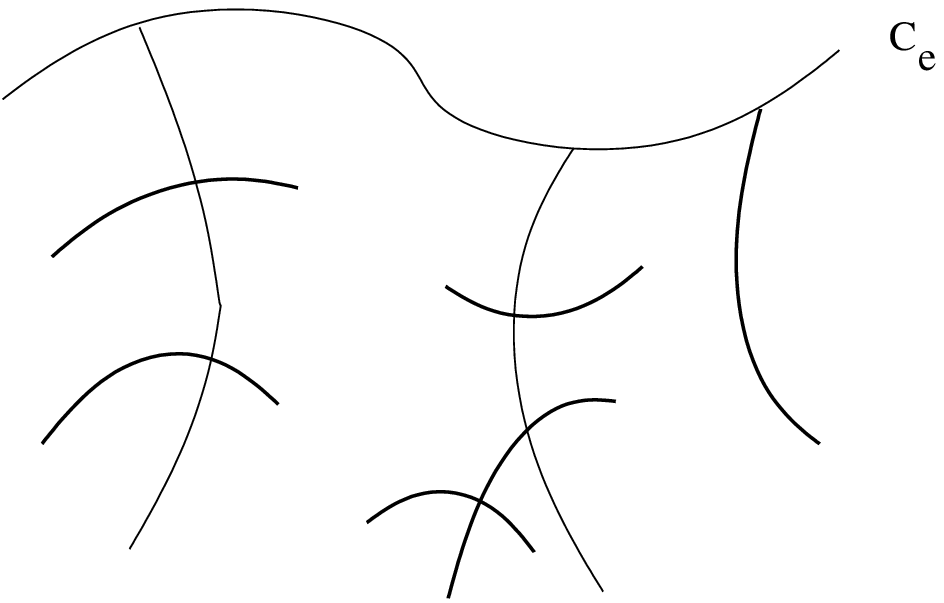, height = 3.5cm,width =
5cm} }
\end{picture}

\vskip .5cm

\centerline{Figure 1.  An elliptic curve $C_e$ attached with 3
rational tails}

\begin{say}
Given any connected genus-one curve $C$, upon removing the core of
$C$, the rest of irreducible components are all rational curves
(i.e., genus zero curves), we denote their union by $C'$. We will
call each connected component of $C'$ a tail, it is a tree of
rational curves. If there are $r$ such connected components, we
will say the curve $C$ has $r$ (rational) tails. See Figure 1 for
an example of elliptic curve with three rational tails.
\end{say}

\begin{say}
By Riemann-Roch, we can check that $R^1\pi_* \sO_Z(mS)=0$ at point
$b \in B$ where the section $S$ meets the core of $Z_b$ but
$R^1\pi_* \sO_Z(mS)$ does not vanish otherwise. In particular,
 $\pi_* \sO_Z(mS)$ is locally
free at point $b \in B$ where the section meets the core of $Z_b$,
 but it is not locally free otherwise.
\end{say}

\begin{say} So we will study examples of families $\pi: Z \lra B$ together with a section $S$
having a point $0 \in B$ such that
\begin{itemize}
\item for $b \ne 0$, the fiber $Z_b$ is smooth; \item the fiber
$Z_0$ is an elliptic curve with $r$ tails; \item  $S$ misses the
core of $Z_0$ but meets the tails.
\end{itemize}
\end{say}

\vskip 9cm

\begin{picture}(3, 5)
\put(90,5){ \psfig{figure=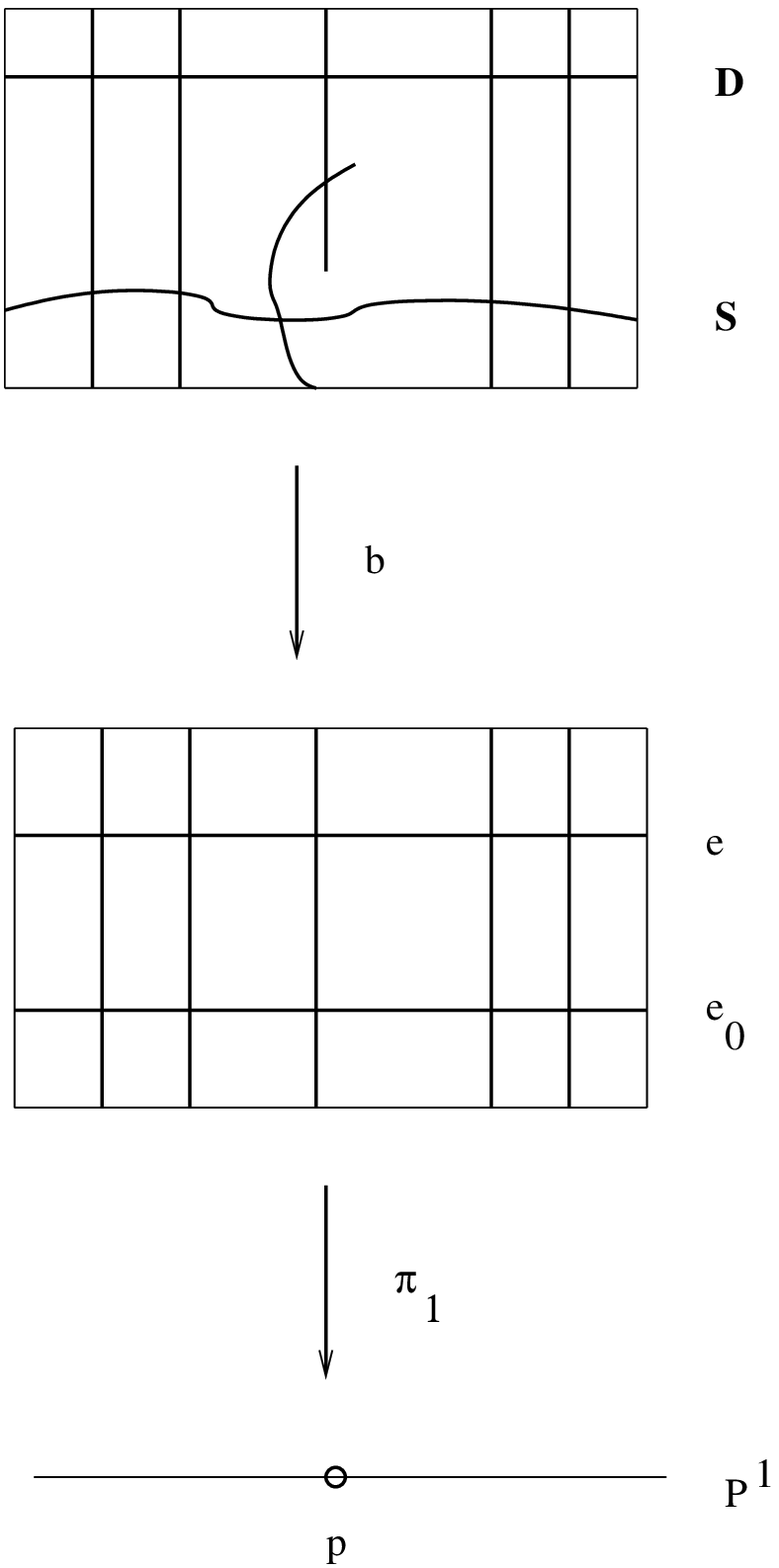, height = 8.5cm,width = 5cm}
}
\end{picture}

\vskip .5cm

\centerline{Figure 2.  The family $Z \lra {\mathbb P}^1 $}

\bigskip\bigskip

\begin{say}
 We will begin with three
toy examples in order of generality: a warm-up 1-tail case, an $
r$-tail case ($r> 1$), and a case of more general type. Throughout
the note, unless otherwise stated, we can work over any fixed
algebraically closed base field $\kk$.
\end{say}

\subsection{A 1-tail case}
\label{1-tail}

%\subsubsection{The main example}
\begin{say} We now begin to construct a one-parameter family of
genus-1 curves whose central fiber is a smooth  elliptic curve
attached with a smooth rational curve. One may consult Figure 2
for picture of such a construction. The details are in the next
two paragraphs.
\end{say}

\begin{say}
Let $b: Z \lra {\mathbb P}^1 \times E$ be the blow up of $
{\mathbb P}^1 \times E$ at $(p,e_0)$ where $p$ a fixed point on
${\mathbb P}^1$ and $e_0$ is a fixed point on $E$, respectively.
Let
$$\pi_1: {\mathbb P}^1 \times E \lra {\mathbb P}^1$$ be the
projection to the first factor and $$\pi= \pi_1 \circ b:  Z \lra
{\mathbb P}^1 \times E \lra {\mathbb P}^1.$$ Note that this
provides a one-parameter smoothing of an elliptic curve with one
rational tail. We will denote the fiber $\pi^{-1}(t)$ by $Z_t$, $t
\in {\mathbb P}^1$.
\end{say}

\begin{say}
Choose a generic point $e \in E$. Let $S, D$ be the proper
transform of ${\mathbb P}^1 \times e_0$ and ${\mathbb P}^1 \times
e$, respectively. Note that if we write the central fiber of $\pi$
as $C_o + C_a$ where $C_o$ is elliptic and $C_a$ is rational, then
$\mathscr O_Z (D) = \mathscr O_Z (S + C_a)$.  We will consider the
direct image sheaf $\mathscr L_m = \pi_*\mathscr O_Z(mS)$.
\end{say}

\begin{say}
We have a short exact sequence
$$0 \lra  \mathscr O_Z(mS) \lra  \mathscr O_Z(mS+D)
\lra   \mathscr O_Z(mS+D)|_{D} \lra 0$$ and  a long exact sequence
\begin{equation}
\begin{CD}
\label{longSD} 0 \lra \pi_* \mathscr O_Z(mS) @>{\alpha_m}>> \pi_*
\mathscr O_Z(mS+D) @>{\beta_m}>> \\  \pi_*   \mathscr
O_Z(mS+D)|_{D} @>{\gamma_m}>> R^1\pi_* \mathscr O_Z(mS)\lra 0.
%0 \lra \pi_* \mathscr O_Z(mS) \lra \pi_*  \mathscr O_Z(mS+D)\lra \pi_*   \mathscr O_Z(mS+D)|_{D} \lra
%R^1\pi_* \mathscr O_Z(mS)\lra 0.
\end{CD}
\end{equation}
Here observe that $\pi_*  \mathscr O_Z(mS+D)$ is locally free
because one  calculates that $$\dim H^0(Z_t, \mathscr
O_Z(mS+D)|_{Z_t}) = m+1, \quad \hbox{ for all} \; t \in {\mathbb
P}^1$$ and $$R\pi_*^1 \mathscr O_Z(mS+D)=0.$$  Since
$$\sO(mS+D)|_D=\sO(D)|_D=N_{D\backslash Z}\cong \sO_D$$ and
$\pi|_D:D\to \Po$ is an isomorphism,   we see that $\pi_* \mathscr
O_Z(mS+D)|_{D}= \mathscr O_{{\mathbb P}^1}.$
\end{say}

\begin{say}
Our goal is to describe explicitly the entire sequence
(\ref{longSD}). \end{say}

\begin{say}
The case $m=0$ is somewhat special, we isolate it below.
$$
0 \lra \pi_* \mathscr O_Z \lra \pi_*  \mathscr O_Z(D) \lra \pi_*
\mathscr O_Z(D)|_{D} \lra R^1\pi_* \mathscr O_Z \lra 0.
$$
It is easy to see that this is
\begin{equation}
\begin{CD}
\label{m=0} 0 \lra \mathscr O_{{\mathbb P}^1} @>{\cong}>> \mathscr
O_{{\mathbb P}^1} @>{0}>>  \mathscr O_{{\mathbb P}^1} @>{\cong}>>
 \mathscr O_{{\mathbb P}^1}  \lra 0.
\end{CD}
\end{equation}
\end{say}

%\subsubsection{Appendix to \S \ref{1-tail}}
%This appendix is not used elsewhere in this note.
%We include it  here for the sharpness and  clarity of the picture in this particular case.

\begin{say}
%Now assume that $m \ge 1$, and
Toward the general case of (\ref{longSD}),
we first  consider another  short exact sequence:
$$0 \lra  \mathscr O_Z(mS+D) \lra  \mathscr O_Z((m+1)S+D)
\lra   \mathscr O_Z((m+1)S +D)|_S \lra 0.$$ Noting that
$$\sO_Z((m+1)S+D)|_S=N_{S\backslash Z}^{\otimes m+1}=\sO_S(-m-1)$$
and $$R^1 \pi_* \mathscr O_Z(mS+D)=0,$$ we obtain a short  exact
sequence of locally free sheaves
\begin{equation}\label{DD}
0 \lra \pi_* \mathscr O_Z(mS+D) \lra \pi_*  \mathscr O_Z((m+1)S+D)
\lra \sO_{{\mathbb P}^1} (-m-1) \lra 0.
\end{equation}
\end{say}

We claim

\begin{lemm}\label{splitDD} The sequence (\ref{DD}) splits
%\footnote{This spliting property depends on the special fact that the base is ${\mathbb P}^1$.}
and consequently
$$\pi_*\mathscr O_Z(mS+D) = \bigoplus_{i=0}^m \sO_{{\mathbb P}^1} (-i).$$
\end{lemm}
\begin{proof}
We  prove it by induction. For $m=0$, it is clear that
$$\pi_* \sO_Z(D) = \sO_{\Po}.$$
%(\ref{DD}) is
%\begin{equation*}
%0 \lra  \sO_{{\mathbb P}^1} \lra \pi_*  \mathscr O_Z(S+D)
%\lra \sO_{{\mathbb P}^1} (-1) \lra 0.
%\end{equation*}
%By Serre duality, we have
%$$\Ext^1(\sO_{{\mathbb P}^1} (-1),
%\sO_{{\mathbb P}^1})= H^0 (\sO_{{\mathbb P}^1} (-3) = 0.$$
%Hence the sequence splits.

Assume that the lemma is true for the case of $m$. Then by
(\ref{DD}), we have
$$0 \lra  \bigoplus_{i=0}^m \sO_{{\mathbb P}^1} (-i) \lra
\pi_*  \mathscr O_Z((m+1)S+D) \lra \sO_{{\mathbb P}^1} (-m-1) \lra
0.$$

Using Serre duality, one calculates that
$$\Ext^1_\Po\bl\sO_{{\mathbb P}^1} (-m-1),
\bigoplus_{i=0}^m \sO_{{\mathbb P}^1} (-i)\br=0,
$$
hence the exact sequence above must be trivial.
\end{proof}

\begin{say}
Let $[t, s]$ be the homogeneous coordinates of  ${\mathbb P}^1$. We may assume that $p=[0,1]$.
Then we have a canonical  map $\beta$,
 \begin{equation}\label{beta}
\begin{CD}
\sO_{{\mathbb P}^1} (-1) @>{\times t}>> \sO_{{\mathbb P}^1} \\
\end{CD}
\end{equation}
which induces a canonical exact sequence
 \begin{equation}\label{CoreSequence}
\begin{CD}
0 @>>> \sO_{{\mathbb P}^1} (-1) @>{\times t}>> \sO_{{\mathbb P}^1} @>{\gamma}>> {\bf k}(p) @>>> 0\\
\end{CD}
\end{equation}
where $\kk(p)$ is the one-dimensional skyscraper sheaf supported
at $p$ and $\gamma$ is the evaluation at $p$. We make an
observation here that any map $\beta': \sO_{{\mathbb P}^1} (-1)
\lra  \sO_{{\mathbb P}^1}$ such that $\gamma \circ \beta' =0$ is a
scalar multiple of $\beta$. This is because $\dim \Hom
(\sO_{{\mathbb P}^1} (-1), \sO_{{\mathbb P}^1}) = 2$ and $\Hom
(\sO_{{\mathbb P}^1} (-1), \sO_{{\mathbb P}^1})$ has a basis $\{
\times t, \times s \}$. In particular, any such nontrivial
$\beta'$ determine the same cokernel, namely, $\kk(p)$.
\end{say}

\begin{prop}
\label{prop1tail}
Assume $m > 0$. Up to isomorphism, we have
\begin{enumerate}
\item  $\pi_* \mathscr O_Z(mS) \cong \bigoplus_{1 \ne i=0}^m
\sO_{{\mathbb P}^1} (-i).$ \item  The map $\alpha_m:\pi_* \mathscr
O_Z(mS) \lra
 \pi_*  \mathscr O_Z(mS+D) \cong \bigoplus_{i=0}^m \sO_{{\mathbb P}^1} (-i) $ is the natural inclusion to the corresponding factors.
\item  The map $\beta_m$ is given by $\beta: \sO_{{\mathbb P}^1}
(-1) \lra \sO_{{\mathbb P}^1}$. \item  The map $\gamma_m$ is the
evaluation at $p$.
 In particular,  all $R^1\pi_* \mathscr O_Z(mS)$ are isomorphic to the one-dimensional skyscraper sheaf $k(p)$ supported at $p$.
\end{enumerate}
\end{prop}
\begin{proof}

To start, observe that for any nonnegative integer $k$, a map from
$\sO_{{\mathbb P}^1}(-k)$ to $ \sO_{{\mathbb P}^1}$ is equivalent to
a map from $\sO_{{\mathbb P}^1}    $ to $ \sO_{{\mathbb P}^1} (k)$,
thus the space of all such maps is $H^0(\Po,\sO_{\Po}(k))$.

Now consider the case $m=1$ first. We have
\begin{equation*}
\begin{CD}
0 \lra \pi_* \mathscr O_Z(S) @>{\alpha_1}>>   \sO_{{\mathbb P}^1}
\oplus     \sO_{{\mathbb P}^1}(-1)
  @>{\beta_1}>>   \sO_{{\mathbb P}^1}  @>{\gamma_1}>>
R^1\pi_* \mathscr O_Z( S)\lra 0 \\
\end{CD}
\end{equation*}
By the observation in the start,   we can express $\beta_1$ as
$$(h_0, h_1) \mapsto c_0 h_0 +   c_1 h_1  $$
where $h_0, h_1 \in  \sO_{{\mathbb P}^1}$ and    $c_i \in H^0(\PP^1, \sO_{\PP^1} (i))$ ($i=0, 1$) are two fixed sections.
Since the cokernel of $\beta_1$ supports at $p$, which  follows from
$$R^1\pi_* \mathscr O_Z( mS)\otimes_{\sO_\Po}\kk(p)=
H^1(\sO_{C_o\cup C_1}(mS))=\kk
$$ by the base change property, hence
$c_0=0$ and $c_1$ is a non-zero constant multiple of $t$.
In particular this also implies that $R^1\pi_* \mathscr O_Z( S)$ is
the one-dimensional skyscraper sheaf $k(p)$ supported at $p$, and
$\gamma_1$ is the evaluation at the point $p$.
This proves the case $m=1$.

For $m \ge 1$, consider the following natural  commutative
diagram of long exact sequence, %{\footnotesize
\begin{equation*}
\begin{CD}
0 \lra \pi_* \mathscr O_Z(mS) @>{\alpha_m}>> \pi_* \mathscr
O_Z(mS+D) \\
@VVV @VVV \\ %@VVV @ VVV  \\
0 \lra \pi_* \mathscr O_Z((m+1)S) @>{\alpha_{m+1}}>> \pi_*
\mathscr O_Z((m+1)S+D)
\end{CD}
\end{equation*}
\begin{equation*}\begin{CD} @>{\beta_m}>>\sO_{{\mathbb P}^1}@>{\gamma_m}>>
R^1\pi_* \mathscr O_Z(mS)\lra 0 \\& &  @VVV @ VVV  \\
@>{\beta_{m+1}}>> \sO_{{\mathbb P}^1} @>{\gamma_{m+1}}>> R^1\pi_*
\mathscr O_Z((m+1)S)\lra 0 \end{CD}
\end{equation*}
 Note that the third downward arrow is an isomorphism. The
diagram gives rise to the following one which we will use soon
\begin{equation*}
\begin{CD}
0 \lra \pi_* \mathscr O_Z(S) @>{\alpha_1}>> \pi_*  \mathscr
O_Z(S+D) \\
 @VVV @VVV  \\
0 \lra \pi_* \mathscr O_Z(m S) @>{\alpha_m}>> \pi_*  \mathscr
O_Z(m S+D)
\end{CD}
\end{equation*}
\begin{equation*}
\begin{CD}
@>{\beta_1}>>   \sO_{{\mathbb P}^1}  @>{\gamma_1}>> R^1\pi_*
\mathscr O_Z(S)\lra 0 \\
&& @VVV @ VVV  \\
@>{\beta_m}>>  \sO_{{\mathbb P}^1}   @>{\gamma_m}>> R^1\pi_*
\mathscr O_Z(mS)\lra 0
\end{CD}
\end{equation*}
 We take splits of all $\pi_*  \mathscr O_Z(kS+D)$ $1 \le k \le
m$ so that all the inclusions $\pi_*  \mathscr O_Z(kS+D)
\hookrightarrow  \mathscr O_Z(mS+D)$ are given by the natural
factor inclusions. By the second square of the last diagram, we
see that the $\beta_m$ restricted to the factor $\sO_{{\mathbb
P}^1} (-1)$ is nontrivial. So, up to a non-zero scalar, we can
assume that it is given by multiplication by $t$.

As in the case of $m=1$, we can  express $\beta_m$ as
$$(h_0, {h_1}, {h_2} \cdots, {h_m})
 \mapsto c_0 h_0 + t  h_1 + c_2 h_2 +
\cdots + c_m h_m$$ for some fixed  $ c_i \in H^0(\Po, \sO_{\Po}(i))$, where
$ h_0, \cdots,  h_m  \in \sO_{\Po}$. Using
the base change property, we see that $R^1\pi_* \mathscr
O_Z(mS)$ is supported at $p$, thus $c_0=0$ and $t \mid c_i$ for $i \ge 2$.
So, we can write $c_i = t a_i$ with $a_i \in H^0(\Po, \sO_{\Po}(i-1)) (i \ge 2)$.

 Hence we obtain that $\ker \beta_m$ is given by
$$\{(h_0, h_1,  h_2,  \cdots, h_m) | -h_1 =   a_2 h_2 +
\cdots + a_m  h_m, h_0, h_2, \cdots, h_m \in   \sO_{{\mathbb P}^1}
\}. $$ This clearly is isomorphic to $\bigoplus_{1\ne i=0}^m
\sO_{{\mathbb P}^1} (-i)$. Therefore, by expressing an arbitrary
element $(h_0, h_1,  h_2,  \cdots, h_m)$ as
$$ (h_0, - (a_2 h_2 +
\cdots + a_m h_m),  h_2,  \cdots, h_{m+1})$$ $$ + (0, h_1 + (a_2
h_2 + \cdots + a_m h_m), 0, \cdots, 0),$$
 we conclude that after a (possibly) new split of
 $$\mathscr O_Z(mS+D) = \bigoplus_{i=0}^m \sO_{{\mathbb P}^1}
 (-i),$$
we can identify
 $\ker \beta_m$  with the direct summand $\bigoplus_{1 \ne i=0}^m \sO_{{\mathbb P}^1} (-i)$
and the map $\beta_m$ is given by the restriction to the summand $\sO_{{\mathbb P}^1} (-1)$ which, in turn,  is given
by multiplying by $t$.

All the rest of the statements follow immediately.
\end{proof}

$\;$

\vskip 9cm

\begin{picture}(3, 5)
\put(90,5){ \psfig{figure=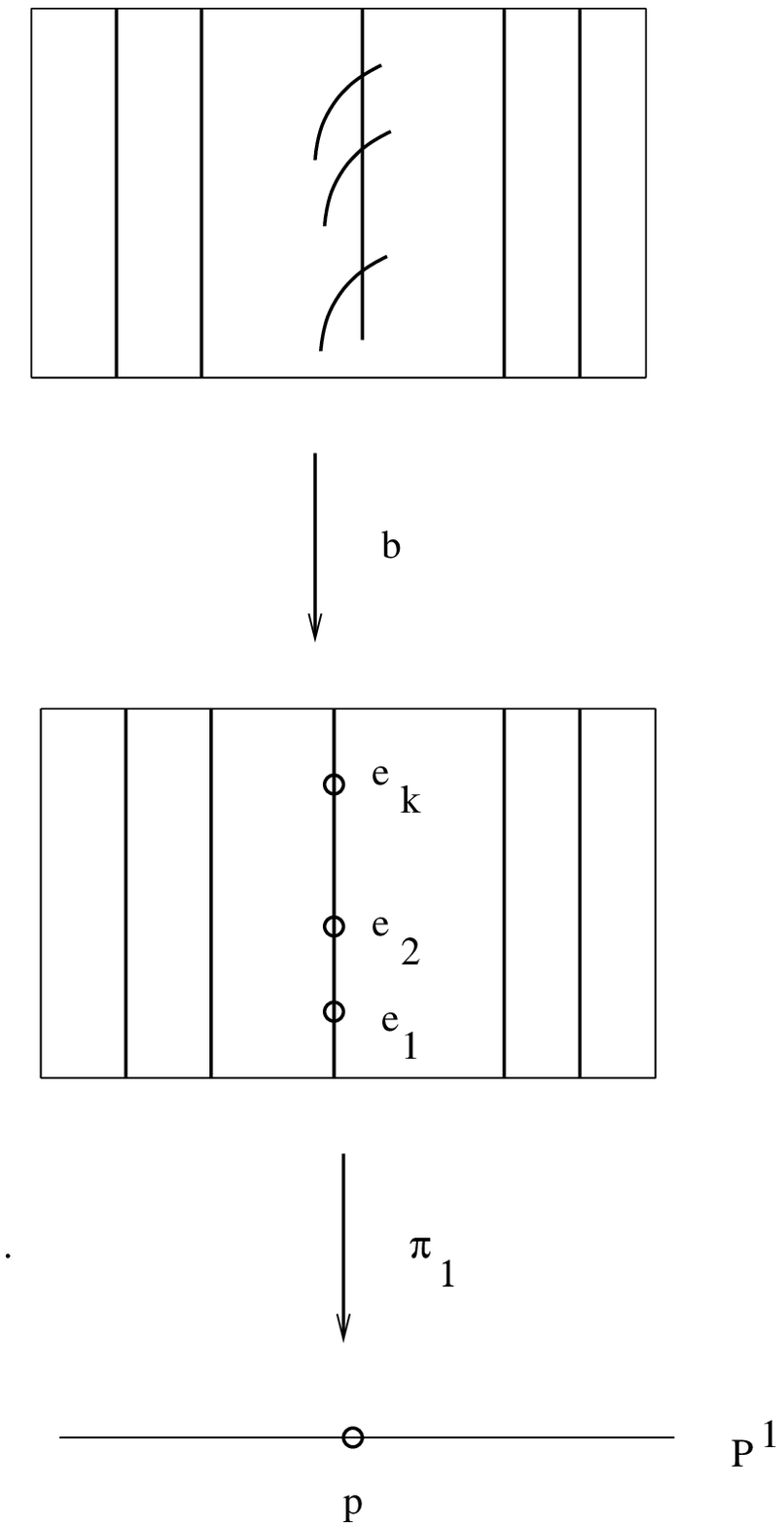, height = 8.5cm,width = 5cm}
}
\end{picture}

%\vskip .5cm

\centerline{Figure 3.  The central fiber is a smooth elliptic with
$k$ tails}

\bigskip%\bigskip

\subsection{Many tails that are smoothed in one direction}
\label{1-tailr}

\begin{say}
We can generalize the above to the case where the central fiber is
an elliptic curve with $k$ many tails. To do this, again on
${\mathbb P}^1 \times E$, we pick up $k$ points $(p, e_1), \cdots,
(p,e_k)$, and blow up ${\mathbb P}^1 \times E$ at  $(p, e_1),
\cdots, (p,e_k)$. This way, we obtain a smooth surface $Z$ whose
projection to  ${\mathbb P}^1$ provides a family $\pi: Z \lra
{\mathbb P}^1 \times E$ with the central fiber an elliptic curve
with $k$ rational tails, elsewhere the fiber is isomorphic to $E$.
In the moduli space, this represents a general direction along
which all the nodes are smoothed simultaneously. See Figure 3.
\end{say}

\begin{say}
Let $S_i$ be the proper transform of ${\mathbb P}^1 \times e_i$,
$1 \le i \le k$ and $D$ be the proper transform of  ${\mathbb P}^1
\times e$ where $e$ is a general point on $E$. Let ${\bf  m} =
(m_1, \cdots, m_k)$ and ${\bf S}=(S_1, \cdots, S_k)$. Let $m =
\sum_i m_i$ and ${\bf mS} = \sum_i m_i S_i$. We consider the
pushforward sheaf $\pi_* \sO_Z ({\bf mS} +D)$.
\end{say}

\begin{say}
Then all the previous results extend almost word by word to this
case with $m$, $S$ replaced by ${\bf  m} $, ${\bf S}$. This
suggests that  the direct image sheaf is sensitive only to the
smoothing direction, but not to the number of tails. The point is
that after carefully treating the case of one tail, it is almost
routine to treat the case of multiple tails.
\end{say}

\begin{say} One can consult \S \ref{rtails} below to see how to formulate similar
statements and arguments. For example, in the inductive proof  in
this case, to go from $m$ to $m+1$, one goes from $\sum_i m_i S_i$
to $(\sum_i m_i S_i) +S_j$. Since the results of this subsection
will not be used elsewhere in this work, we omit the routine
details.
\end{say}

\subsection{An $r$-tail case}
\label{rtails}

\begin{say}
We will now construct a family $Z \lra B$ such that the central
fiber is a smooth elliptic curve attached with a connected chain
of rational curves.
\end{say}

\vskip 7cm

\begin{picture}(3, 5)
\put(90,5){ \psfig{figure=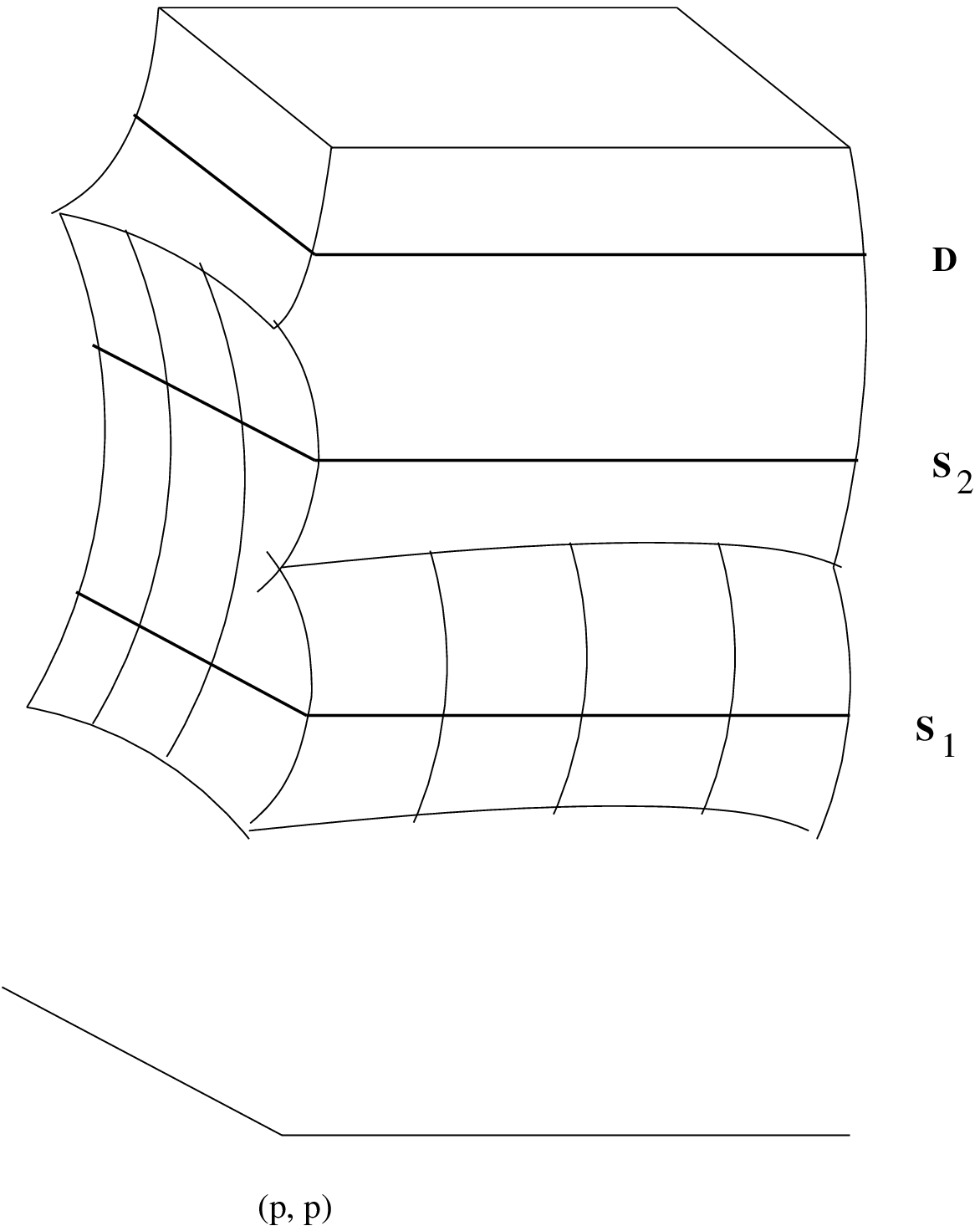, height = 6.5cm,width
= 5cm} }
\end{picture}

%\vskip .5cm

\centerline{Figure 4.  The central fiber has two rational
components}

\bigskip%\bigskip

\begin{say}
Consider the space $(\PP^1)^r \times E$. Let $p = [0,1] \in
\PP^1$. Pick $r$ distinct points $e_1, \cdots, e_r$ in $E$. For
any $1 \le i \le r$, Set $$W_i = \PP^1 \times \cdots \times  \PP^1
\times p \times  \PP^1 \times \cdots   \times \PP^1  \subset
(\PP^1)^r $$ where $p$ occurs in the i-th factor. Set $W_i' = W_i
\times e_i \subset  (\PP^1)^r \times E$, $1 \le i \le r$. Let $b:
Z \lra (\PP^1)^r \times E$ be the blowup of  $(\PP^1)^r \times E$
along the disjoint union $\coprod_i W_i'$. Let $S_i$ be the proper
transform of  $(\PP^1)^r \times e_i$ and $D$ be the proper
transform of  $(\PP^1)^r \times e$ where $e \in E$ is a point
distinct from all $e_1, \cdots, e_r$.
\end{say}

\begin{say}
Let $\pi$ be the  projection  $b: Z \lra (\PP^1)^r \times E$
followed by the projection to  $(\PP^1)^r$. Then this provides a
family of nodal elliptic curves: the  fiber over $\bar{p}:=(p,
\cdots, p)$ is a smooth elliptic curve with $r$ many rational
tails; moving to a general direction represented by coordinates
indexed by $i_1, \cdots, i_j$ ($1 \le j \le r$), the nodes created
by blowing up $W_{i_1}', \cdots, W_{i_j}'$ will be smoothed.
\end{say}

\begin{say}
Figure 4 shows the case of $r=2$. Here, the central fiber over
$(p,p)$ is $C_e\cup C_1 \cup C_2$ where $C_e \cong E$ and $C_1
\cong C_2 \cong \Po$; the section $D$ passes $C_e$; the sections
$S_i$ passes $C_i$ ($i=1,2$).
\end{say}

\begin{say} Let $C_j$ be the exceptional divisor of $Z$  corresponding
to $W_j'$. Observe that $\sO_Z(S_j + C_j) = \sO(D)$. Let ${\bf  m}
= (m_1, \cdots, m_r)$ and ${\bf S}=(S_1, \cdots, S_r)$. Let $m =
\sum_i m_i$ and ${\bf mS} = \sum_i m_i S_i$. We will study the
local freeness of  the direct image sheaf $\pi_* \sO_Z ({\bf
mS})$.
\end{say}

\begin{say}
We have a short exact sequence
\begin{equation}\label{shortSDr}
\begin{CD}
0 \lra  \mathscr O_Z(\bfmS) @>>>  \mathscr O_Z(\bfmS+D) @>>>
\mathscr O_Z(\bfmS+D)|_{D} \lra 0
\end{CD}
\end{equation}
and  a long exact sequence
\begin{equation}
\begin{CD}
\label{longSDr} 0 \lra \pi_* \mathscr O_Z(\bfmS) @>{\alpha_{\bf
m}}>> \pi_*  \mathscr O_Z(\bfmS+D) @>{\beta_{\bf m}}>>  \\ \pi_*
\mathscr O_Z(\bfmS+D)|_{D} @>{\gamma_{\bf m}}>>R^1\pi_* \mathscr
O_Z(\bfmS)\lra 0.
\end{CD}
\end{equation}
As in \S \ref{1-tail},  $\pi_*  \mathscr O_Z(\bfmS+D)$ is locally
free and $R\pi_*^1  \mathscr O_Z(\bfmS+D)=0.$
\end{say}

\begin{say} Since $\sO_Z(\bfmS+D)|_D=\sO_Z(D)|_D=N_{D\backslash Z}\cong
\sO_D$ and $\pi|_D:D \to (\PP^1)^r$ is an isomorphism,   we see
that $\pi_* \mathscr O_Z(\bfmS+D)|_{D}= \mathscr O_{(\PP^1)^r}.$
\end{say}

\begin{say}
The case of (\ref{longSDr}) when $m=0$ is special, we isolate it below.
$$
0 \lra \pi_* \mathscr O_Z \lra \pi_*  \mathscr O_Z(D) \lra \pi_*
\mathscr O_Z(D)|_{D} \lra R^1\pi_* \mathscr O_Z \lra 0.
$$
It is easy to see that this is
\begin{equation}\label{bfm=0}
\begin{CD}
0 \lra \mathscr O_{(\Po)^r} @>{\cong}>> \mathscr O_{(\Po)^r}
@>{0}>> \mathscr O_{(\Po)^r} @>{\cong}>>
 \mathscr O_{(\Po)^r}  \lra 0.
\end{CD}
\end{equation}
\end{say}

\begin{say}
To treat the general case,  similar to  \S  \ref{1-tail}  of the
1-tail case, we will first prove a formula for $\pi_* \mathscr
O_Z(\bfmS+D)$.
\end{say}

We begin with a lemma.

\begin{lemm}\label{OWj}
\label{OW=O1} $\pi_* N_{S_j \backslash Z}$ is  isomorphic to
$\sO_{(\Po)^r} (-W_j)$.  And both are isomorphic to $\pi_j^* \sO_{
\Po} (-1)$
%$$ \pi_1^* \sO_{\Po}  \otimes \cdots \otimes  \pi_{j-1}^* \sO_{ \Po}  \otimes  \pi_j^* \sO_{ \Po} (-1) \otimes  \pi_{j+1}^* \sO_{ \Po}  \otimes
%\cdots \otimes \pi_r^* \sO_{ \Po} $$
where $\pi_j: (\Po)^r \lra \Po$ is the projection to the j-th
factor.
\end{lemm}
\begin{proof}
Any bundle on $(\Po)^r$ is of the form $\bigotimes_{j=1}^r  \pi_j^* \sO_{ \Po}(a_j)$. We will determine the integer
$a_j$ for each of  $\pi_* N_{S_j \backslash Z}$ and  $\sO_{(\Po)^r} (-W_j)$.

First, it is easy to determine that $\sO_{(\Po)^r} (W_j) = \pi_j^* \sO_{\Po} (1)$ by looking at the intersection numbers of
of $W_j$ with the coordinate lines.
%$$\sO_{(\Po)^r} (W_j) = \pi_1^* \sO_{\Po}  \otimes \cdots
%\otimes  \pi_{j-1}^* \sO_{\Po} \otimes  \pi_j^* \sO_{\Po} (1) \otimes  \pi_{j+1}^* \sO_{ \Po}  \otimes
%\cdots \otimes \pi_r^* \sO_{ \Po} .$$
 Hence $\sO_{(\Po)^r} (-W_j) =\pi_j^* \sO_{\Po} (-1)$, as desired.

To show the rest, again we can argue  by computing intersection
numbers.  Let $l_i$ be the pre-image of the i-th coordinate line
by the isomorphism  $\pi|_{S_j}: S_j \stackrel{\cong}{\lra}
(\Po)^r$. Then we have  $S_j \cdot l_i = 0$ when $i \ne j$ because
$l_i$ can be moved to become a section of $C_j \cong W_j \times
\Po \lra W_j$, disjoint from $S_j$. When $i=j$, we have
$$(S_j + C_j) \cdot l_j = D \cdot l_j = 0.$$
Hence $S_j \cdot l_j = - C_j \cdot l_j$. But $C_j \cdot l_j=1$. Hence  $\pi_* N_{S_j \backslash Z} = \pi_j^* \sO_{\Po} (-1)$
because  $N_{S_j \backslash Z} = \sO_Z (S_j)|_{S_j}$.

We may also argue without calculating the intersection numbers.
Since $S_j$ is the proper transform of the  constant section
$(P^1)^r \times e_j$ in $(P^1)^r \times E$,  we have that the
normal bundle $N_{S_j \backslash Z}$ is isomorphic to
$$\sO_{S_j}(-\hbox{exceptional divisor in}\; S_j)$$ where the
exceptional divisor in $S_j$ is the pre-image of $W_j$ in $S_j$.
This implies that $N_{S_j \backslash Z} \cong \sO_{(\Po)^r}
(-W_j)$.
\end{proof}

\begin{say}
Now,
consider the  short exact sequence:
$$0 \lra  \mathscr O_Z(\bfmS+D) \lra  \mathscr O_Z(\bfmS+S_j+ D)
\lra   \mathscr O_Z(\bfmS +S_j+ D)|_S \lra 0.$$ Noting that we
have  $$\sO_Z(\bfmS+S_j+D)|_{S_j}=N_{S_j\backslash Z}^{\otimes
(m_j+1)},$$
$$\pi_* N_{S_j\backslash Z}^{\otimes
(m_j+1)} = \sO_{(\Po)^r}(-(m_i+1) W_j) \;\; \hbox{(Lemma
\ref{OWj})},$$ and
$$R^1 \pi_* \mathscr O_Z(mS+D)=0,$$ hence we obtain   a
short  exact sequence of locally free sheaves
\begin{equation}\begin{CD}\label{DDr}
0 \lra \pi_* \mathscr O_Z(\bfmS+D) \lra \pi_* \mathscr
O_Z(\bfmS+S_j+D) \\ \lra  \sO_{(\Po)^r}(-(m_j+1) W_j) \lra 0.
\end{CD}
\end{equation}
\end{say}

\begin{lemm}
\label{lemmSplit}
 The sequence (\ref{DDr}) splits and consequently
$$\pi_*  \mathscr O_Z(\bfmS+D) = \sO_{(\Po)^r} \oplus \bigoplus_{i=1}^r \bigoplus_{k=1}^{m_i} \sO_{(\Po)^r} (-kW_j). $$
\end{lemm}
\begin{proof}
We proceed by induction on $\sum m_i=m$. When $m=0$, we have
$$\pi_* \sO_Z(D) = \sO_{(\Po)^r}.$$
%(\ref{DDr}) is
%\begin{equation*}
%0 \lra  \mathscr O_{(\Po)^r} \lra \pi_*  \mathscr O_Z(S_j+D)
%\lra  \sO_{(\Po)^r}(- W_j) \lra 0.
%\end{equation*}
%Using  Serre's duality and $\sO_{(\Po)^r}(- W_j)= \pi_j^* \sO_{\Po} (-1)$, we have
%$$\Ext^1( \sO_{(\Po)^r} (- W_j),  \mathscr O_{(\Po)^r}) = 0.$$
Hence the lemma holds in this case.

Assume that when  $\sum m_i=m$, the lemma holds. Then for the case
of $m+1$,  by (\ref{DDr}), we have
\begin{equation*}
\begin{CD}
0 \lra \sO_{(\Po)^r} \oplus \bigoplus_{i=1}^r
\bigoplus_{k=1}^{m_i} \sO_{(\Po)^r} (-kW_j) \lra \pi_*
\mathscr O_Z(\bfmS+S_j+D) \\
\lra  \sO_{(\Po)^r}(-(m_j+1) W_j) \lra 0.
\end{CD}
\end{equation*}
Now applying  Serre's duality and Lemma, we obtain
 $$\Ext^1 (\sO_{(\Po)^r}(-(m_j+1) W_j),   \sO_{(\Po)^r} \oplus \bigoplus_{i=1}^r \bigoplus_{k=1}^{m_i} \sO_{(\Po)^r} (-kW_j))=0, $$
 hence the sequence is the trivial one.
\end{proof}

\begin{say} Our aim is to explicitly describe the sheaf $\pi_* \mathscr
O_Z(\bfmS)$. By (\ref{bfm=0}), we may assume $m >0$. Let  $[t_i,
s_i]$ be  the homogeneous coordinates of  the i-th factor
${\mathbb P}^1$ in the product $(\PP^1)^r$. Then
$$W_i = \{ t_i =0 \}.$$
Note that $\cap_i W_i = \bar{p}$.
Let $$\sV_1 = \bigoplus_{i=1}^r \sO_{(\PP^1)^r} (-W_i)= \bigoplus_i \pi_i^* \sO_{\PP^1}(-1).$$
Then we have a canonical  map $\beta$,
 \begin{equation}\label{betar}
\begin{CD}
\sV_1  @>>> \mathscr O_{(\PP^1)^r} \\
(h_1, \cdots, h_r)  @>>> \sum_i t_i h_i
\end{CD}
\end{equation}
which induces a canonical exact sequence
 \begin{equation}\label{CoreSequencer}
\begin{CD}
0 \lra  \ker \beta  @>>> \sV_1  @>{\beta}>> \sO_{(\PP^1)^r}
@>{\gamma}>>   {{\sO_{(\PP^1)^r}}}|_W = \kk(\bar{p}) \lra 0
\end{CD}
\end{equation}
where $\kk(\bar{p})$ is the structure sheaf of $\bar{p}$.
This sequence is a higher dimensional generalization of (\ref{beta}) and plays similar role in the proof.
\end{say}

Set $$\sV_{{\bf m},0} =  \sO_{(\Po)^r} \oplus \bigoplus_{ j=1}^r \bigoplus_{k=2}^{m_j} \sO_{(\Po)^r} (-kW_j).$$

\begin{prop}
\label{propRtail}
 Assume\footnote{When some $m_j =0$, we can ignore $W_j$, $S_j$, and the family $C_j$ from the consideration, this tail will not affect
the sheaf $\pi_* \mathscr O_Z(\bfmS)$, and the problem is reduced to the previous $(r-1)$-tail case. This is analogous to
the situation of \S \ref{1-tailr}.}
 that $m_j \ne 0$ for all $1 \le j \le r$.  Then up to isomorphism, we  have  a commutative diagram
\begin{equation}
\label{rtailmain}
\begin{CD}
0 \lra  \ker \beta  @>>> \sV_1  \\
@VVV @VVV\\ 0 \lra \pi_* \mathscr O_Z(\bfmS) @>{\alpha_{\bf m}}>>
\pi_* \mathscr O_Z(\bfmS+D)
\end{CD}
\end{equation}
\begin{equation*}
\begin{CD}
@>{\beta}>> \sO_{(\PP^1)^r} @>{\gamma}>>   \kk(\bar{p}) \lra 0 \\
&&  @VVV  @VVV  \\
 @>{\beta_{\bf m}}>>  \sO_{(\PP^1)^r}
@>{\gamma_{\bf m}}>>    R^1\pi_* \mathscr O_Z(\bfmS) \lra 0
\end{CD}
\end{equation*}
where the first two downward arrows are inclusions, and the last
two are isomorphisms.  Moreover, we have
\begin{enumerate}
\item $\pi_*  \mathscr O_Z(\bfmS+D) = \sV_{{\bf m},0}  \oplus  \sV_1$.
\item $\pi_* \mathscr O_Z(\bfmS) \cong \ker \beta_{\bf m} = \sV_{{\bf m},0}  \oplus \ker \beta $.
\item $\beta_{\bf m} |_{\sV_1} = \beta$.
\end{enumerate}
\end{prop}
\begin{proof} The proof is totally analogous to the proof given in  Proposition \ref{prop1tail} with main  exception
 that $\ker \beta$ needs not to be trivial in this higher dimensional case.  Below, we will provide enough details.

(1) is just a rephrase of Lemma \ref{lemmSplit}. We will prove the rest all together.

The initial case is when all $m_j = 1$, $1 \le j \le r$. Denote
$(1, \cdots, 1) \in {\mathbb Z}^r$ by ${\bf 1}$. We have  $$\pi_*
\mathscr O_Z({\bf 1 \cdot S}+D) =\sO_{(\Po)^r} \oplus  \sV_1.$$
Hence
\begin{equation*}
\begin{CD}
0 \lra \pi_* \mathscr O_Z(\bf 1 \cdot  S) @>{\alpha_{\bf 1}}>>
\sO_{(\Po)^r} \oplus  \sV_1
 @>{\beta_{\bf 1}}>>  \sO_{(\PP^1)^r} \\
  @>{\gamma_{\bf 1}}>>    R^1\pi_* \mathscr O_Z(\bf 1\cdot  S)    \lra
  0.
\end{CD}
\end{equation*}
So, we  can express $\beta_{\bf 1}$ as
$$(h_0, h_1, \cdots, h_r) \mapsto c_0 h_0 +  \sum_i  c_i h_i  $$
where $h_0, h_1 \in  \sO_{(\Po)^r}$, $c_0$ is a fixed scalar, and
$c_i (t_i,s_i) \in H^0(\PP^1, \sO_{\PP^1} (1))$ are fixed
sections. The same argument as in the case of $r=1$ implies that
$c_0=0$ and $t_i \mid c_i$. By changing coordinates if necessary,
we can write  $c_i = t_i$. Then $(h_0, h_1, \cdots, h_r) \in \ker
\beta_{\bf 1}$ if and only if $\sum_i  t_i h_i =0$. That is, $\ker
\beta_{\bf 1} = \ker \beta \oplus \sO_{(\Po)^r}$. Hence the case
${\bf m} = {\bf 1}$ of the lemma follows.

When ${\bf m} \ge {\bf 1}$, meaning all $m_j \ge 1$, we have
\begin{equation*}
\begin{CD}
0 \lra \pi_* \mathscr O_Z(\sum_i S_i) @>{\alpha_{\bf 1 }}>> \pi_*
\mathscr O_Z(\sum_i S+i+ D) \\
@VVV @VVV \\ 0 \lra \pi_* \mathscr O_Z(\bfmS) @>{\alpha_{\bf m}}>>
\pi_* \mathscr O_Z(\bfmS+D)
\end{CD}
\end{equation*}
\begin{equation*}
\begin{CD}
@>{\beta_{\bf 1}}>>   \sO_{{\mathbb P}^1}  @>{\gamma_{\bf 1}}>>
R^1\pi_* \mathscr O_Z(\sum_i S_i)\lra 0 \\
&& @VVV @ VVV  \\
@>{\beta_{\bf m}}>>  \sO_{{\mathbb P}^1} @>{\gamma_{\bf m}}>>
R^1\pi_* \mathscr O_Z(\bfmS)\lra 0
\end{CD}
\end{equation*}
 Hence we can assume that $\beta_{\bf m}|_{V_1} = \beta$. Then
the map $\beta_{\bf m}$ is given by
$$
\begin{CD}
(h_0,  (h_1, \cdots, h_r), h_{12}, \cdots, h_{1m_1}, \cdots,
h_{r2}, \cdots, h_{rm_r})  \\
@VVV \\
 c_0 h_0 + \sum_{j=1}^r t_j h_j + \sum_{j=1}^r
\sum_{k=2}^{m_j} c_{jk} h_{jk}.
\end{CD}
$$
 As in Proposition
\ref{prop1tail}, $c_0 = 0$ and we may write $c_{jk} = t_j a_{jk}$
($ 1 \le j \le r, 2 \le k \le m_j$). Hence $(h_0,  h_1, \cdots,
h_r, h_{12}, \cdots, h_{1m_1}, \cdots,  h_{r2}, \cdots, h_{rm_r})
\in \ker \beta_{\bf m}$ if and only if
$$\sum_{j=1}^r t_j (h_j+  \sum_{k=2}^{m_j} a_{jk} h_{jk}) = 0.$$
We can split an arbitrary point
 $$(h_0,  (h_1, \cdots, h_r), h_{12}, \cdots, h_{1m_1}, \cdots,  h_{r2}, \cdots, h_{rm_r})$$
as the sum of
$$(h_0,  (\cdots,  -\sum_{k=2}^{m_j} a_{jk} h_{jk}, \cdots), h_{12}, \cdots, h_{1m_1}, \cdots,  h_{r2}, \cdots, h_{rm_r})$$
and
$$(0,  (\cdots,  h_j+\sum_{k=2}^{m_j} a_{jk} h_{jk}, \cdots), 0, \cdots, 0).$$
This gives a splitting of $\pi_*  \mathscr O_Z(\bfmS+D)$ as the direct sum of $\sV_{{\bf m},0}$ and $\sV_1$. Under this direct sum,
we see that $\ker \beta_{\bf m} = \sV_{{\bf m},0} \oplus \ker \beta$. The rest follows straightforwardly.
\end{proof}

\begin{say}
Clearly,  $\ker \beta$,  independent of ${\bf m}$,
 is the sole cause for non-local freeness of the sheaf $\pi_* \mathscr O_Z(\bfmS)$. Blowup is to make
$\ker \beta$ locally free, thus resolves the sheaf $\pi_* \mathscr
O_Z(\bfmS)$.
\end{say}

\begin{rema} \begin{enumerate} \item  Note here that the essence of the arguments on the sheaf is independent of the genus of the family.
\item Also, the decomposition $\sV_1 = \bigoplus_i \pi_i^*
\sO_{\PP^1}(-1)$ seems to suggest  a relation with either
restricting the family to the coordinate directions or projections
to the coordinate directions.
\end{enumerate}
\end{rema}

\begin{say}
Let $\varphi: \widetilde{(\Po)^r} \lra (\Po)^r$ be the blowup at
the point $\bar{p}=(p, \cdots, p)$. We have the square
\begin{equation}
\begin{CD}
\tilde{Z} @>{\psi}>> Z \\
 @V{\tilde{\pi}}VV       @V{\pi}VV    \\
\widetilde{(\Po)^r} @>{\varphi}>> (\Po)^r.
\end{CD}
\end{equation}
\end{say}

\begin{prop}
The sheaf $\tilde{\pi}_* \psi^* \mathscr O_Z(\bfmS)$ is locally free.
\end{prop}
\begin{proof}
This can be checked directly. (See \cite{HL08} for a general
proof. The reader is encouraged to do this example.)
\end{proof}

\subsection{A further example} %The direct image sheaf around deeper strata: an example}
\label{deepExample}
\begin{say}
 To gain further intuition  around deeper
strata of the moduli space $\MPd$, let us look at the stratum of
the type $o[a[b,c]]$.  $o[a[b,c]]$ is the square bracket notation
for the first tree of Figure 5. (The interested reader is referred
to \cite{HL08} for further explanation of this notational scheme;
he is also encouraged  to construct a family for the second curve
in Figure 5.)  Here, any curve in the stratum consists of a smooth
elliptic curve $C_o$, a rational curve $C_a$ attached to $C_o$,
and two more rational curves $C_b$ and $C_c$ attached to $C_a$. A
line bundle over the family of the interest is trivial on both
$C_o$ and $C_a$, and positive on $C_b$ and $C_c$.
\end{say}

We will construct explicitly such a family $Z$ over a three
dimensional base $B$. We work over $\mathbb C$ in this subsection.

\begin{say}
Let $Z_a$ be the family of nodal elliptic curve constructed in \S
\ref{1-tail}. That is, $Z_a$ is the blowup of $\Po \times E$ at
$(p,e_0)$. We would like to have two disjoint sections that go
through the rational curve of the central fiber. So, we take a
small analytic disc $\Delta$ of the base center at $p$.
%For simplicity, we still
We denote the family restricted to $\Delta$ by $Z_a|_\Delta$. So,
let $S_{a,i}$ be two disjoint sections of  $Z_a|_\Delta$ that go
through the rational curve of the central fiber, $i=b,c$. Let
$D_a$ be the proper transform of  $\Po \times e$ in $Z_a$, $e \ne
e_0$.
\end{say}

\begin{say}
Consider $Z_a|_\Delta \times \Po \times \Po$. Let $W_b = S_{a,b}
\times p \times \Po $ and  $W_c =  S_{a,c} \times \Po \times p $.
Let $Z$ be the blowup of  $Z_a|_\Delta \times \Po \times \Po$
along $W_b \coprod W_c$.

Then we obtain a family of nodal elliptic curve over  $B=\Delta
\times \Po \times \Po$
$$Z \lra B$$
such that: moving along the first coordinate direction, the node
$a$ is smoothed, but the elliptic curve comes with two solid
tails; moving along the second direction, the node $b$ is
smoothed;
 moving along the last direction, the node $c$ is smoothed.
\end{say}

\vskip 8cm

\begin{picture}(3, 5)
\put(30,5){ \psfig{figure=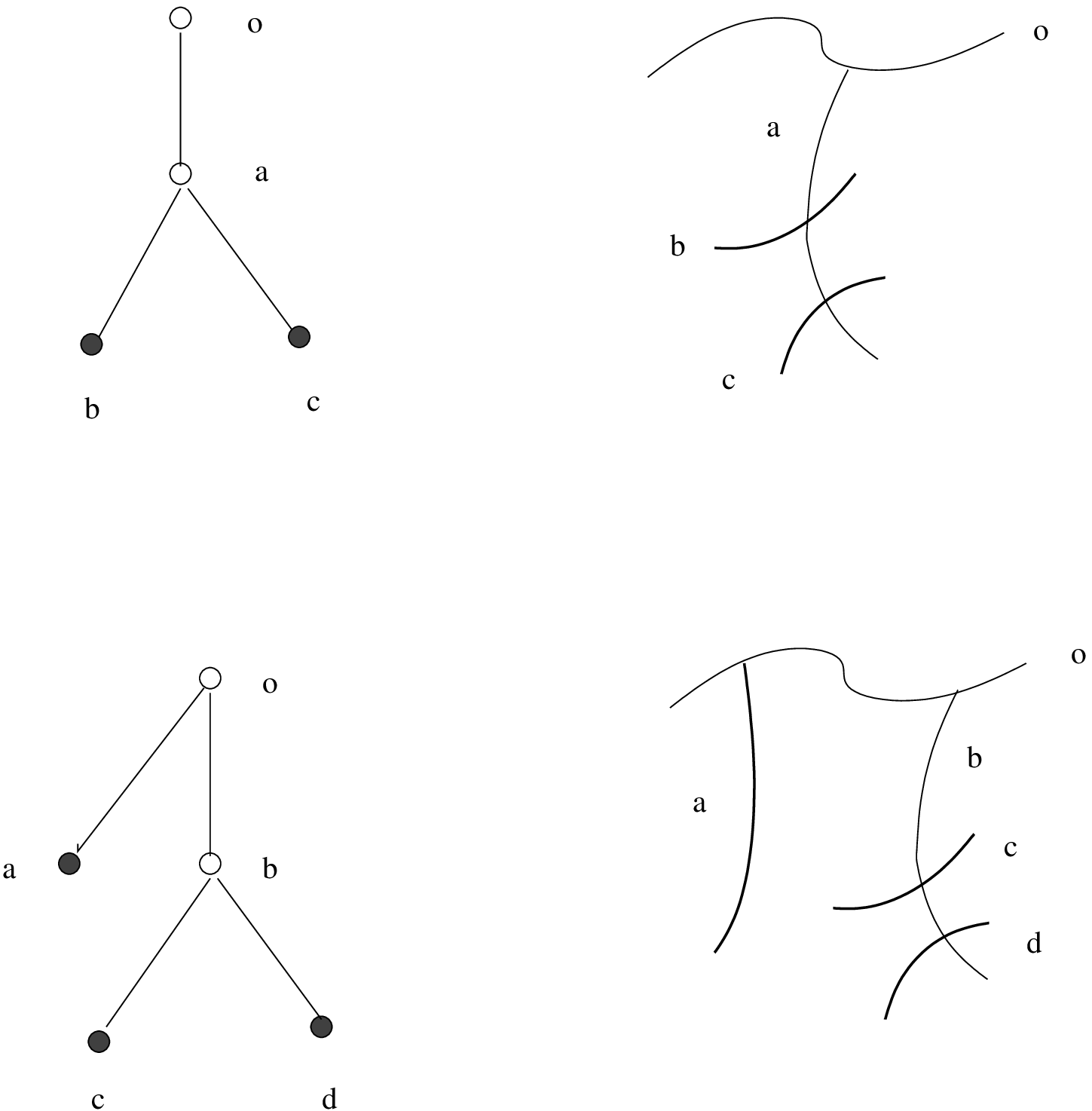, height = 7.5cm,width =
9cm} }
\end{picture}

\vskip .5cm

\centerline{Figure 5.  Two elliptic curves and two trees}

\vskip .5cm

\begin{say}
Let $D$ be the proper transform of $D_a|_\Delta \times \Po \times
\Po$ in $Z$, $S_b$  the proper transform of $S_{a,b} \times \Po
\times \Po$, and $S_c$ the proper transform of $S_{a,c} \times \Po
\times \Po$.
\end{say}

Let $V_a = p \times \Po \times \Po$, $V_b = p \times \Po \times \Po$, $V_c=  \Po \times p \times \Po$.

Let $\pi_i$ be the projection of $\Delta \times \Po \times \Po$ to the i-th factor, $i=a,b,c$.

A main new feature in this example is

\begin{lemm}
$$\pi_* \sO_Z(S_b)|_{S_b} = \sO_B (-V_b-V_a) =\pi_a^* \sO_\Po (-1) \otimes \pi_b^* \sO_\Po (-1).$$
$$\pi_* \sO_Z(S_c)|_{S_c} = \sO_B (-V_c-V_a)=\pi_a^* \sO_\Po (-1) \otimes \pi_c^* \sO_\Po (-1) .$$
\end{lemm}
\begin{proof}
Both can be checked  directly by calculating the degree on each coordinate direction.
\end{proof}

Let $\bar{p}=(p,e_0,p,p) \in B$. Then, almost identical arguments
as in the sections \S\S \ref{1-tail} and \ref{rtails} will lead to
a canonical homomorphism
\begin{equation*}\label{betaS}
\begin{CD} \sO_B (-V_b-V_a) \oplus \sO_B (-V_c-V_a)  @>{\beta}>> \sO_B \\
(h_a, h_b,h_c) @>>> (t_at_b) h_b+ (t_at_c) h_c =
t_a(t_bh_b+t_ch_c).
\end{CD}
\end{equation*}

Hence we have,
\begin{prop}
Up to isomorphism, we  have  a commutative diagram %{\footnotesize
\begin{equation*}
\begin{CD}
0 \lra  \ker \beta  @>>> \sO_B (-V_b-V_a) \oplus \sO_B (-V_c-V_a)
 \\
@VVV @VVV\\ 0 \lra \pi_* \mathscr O_Z(S_b + S_c) @>{\alpha}>>
\pi_* \mathscr O_Z(S_b + S_c+D)
\end{CD}
\end{equation*}
\begin{equation*}
\begin{CD}
@>{\beta}>> \sO_B @>{\gamma}>>   \kk(\bar{p}) \lra 0 \\
&& @VVV  @VVV  \\
@>{\beta_S}>>  \sO_B @>>> R^1\pi_* \mathscr O_Z(S_b + S_c) \lra 0
\end{CD}
\end{equation*}
 where the first two downward arrows are inclusions, and the last
two are isomorphisms.  Moreover, we have
\begin{enumerate}
\item $\pi_*  \mathscr O_Z(S_b + S_c+D) = \sO_B  \oplus \sO_B
(-V_b-V_a) \oplus \sO_B (-V_c-V_a)$. \item $\pi_* \mathscr O_Z(S_b
+ S_c) \cong \ker \beta_S = \sO_B  \oplus \ker \beta $. \item
$\beta_S |_{\sO_B (-V_b-V_a) \oplus \sO_B (-V_c-V_a)} = \beta$.
\end{enumerate}
\end{prop}
\begin{proof}
Parallel to to the proofs in \S\S \ref{1-tail} and \ref{rtails}.
\end{proof}

\begin{say}
In the next two subsections \S\S \ref{generalCase-r} and
\ref{generalStrata}, we will give interpretations of the main
results of \cite{HL08} on the direct image sheaves in the spirit
of the examples that the reader has already been through. We refer
the proofs to \cite{HL08}.
% though it would be nice to prove these results along the  lines of the examples.
\end{say}

\subsection{The direct image sheaf on  the moduli space:  simple case}
\label{generalCase-r}

\begin{say} Let $\pi: \cX \lra V$ be a flat family of nodal genus
1 curves. Let $0 \in V$ be a fixed point with the fiber
$$\cX_0 = \bigcup_{i=0}^r C_i$$ such that
$C_0$ is a nodal elliptic curve, $C_i$ are smooth rational curves
attached to $C_0$ at distinct smooth points.
\end{say}

\begin{say}
We will first make sense of the subscheme $V_i\sub V$ that are
locus of the $i$-th tail not smoothed ($1 \le i \le r$). We let
$p: \cX\to\cY$ be the stabilization\footnote{In this case, this
means that the morphism $p$ contracts any rational components that
has fewer than three nodes.}, let $\cE\sub\cX$ be the exceptional
divisors, let $\cE_i$ be the irreducible component of $\cE$  that
contains $C_i\sub\cX_0$. We define
$$V_i=\pi(\cE_i).
$$
\end{say}

\begin{lemm} The subset $V_i\sub V$ is canonically a closed subscheme whose
ideal sheaf $\sI_{V_i\sub V}$ is generated by a regular function
$t_i\in\Gamma(\sO_V)$.
\end{lemm}

So we can write $V_i =\{t_i=0\}$.  For every $1 \le i \le r$, let
$\cX_i$ be the restriction of the total family to $V_i$.  Choose a
section $S_i$ over $V$ that meets $\cE_i$ ($1 \le i \le r$), and a
generic section $D$ that misses all $\cE_i$. All these can be done
by shrinking $V$ if necessary.

\begin{say}
Again we will consider sheaves of the form
 $$\pi_* \sO_\cX (\bfmS), \quad {\bf m} \ge {\bf 1}.$$
We have a short exact sequence
$$0 \lra  \sO_\cX(\bfmS)  \lra  \sO_\cX (\bfmS +D)
\lra    \sO_\cX(\bfmS +D)|_{D} \lra 0 $$ and  a long exact
sequence
\begin{equation}\label{mainSequence-R-tail}
\begin{CD}
0 \lra \pi_* \sO_\cX(\bfmS) \lra \pi_*
 \sO_cC(\bfmS +D) \lra \\ \pi_*    \sO_\cX(\bfmS +D)|_{D}
  \lra
R^1\pi_* \sO_\cX(\bfmS) \lra 0.
\end{CD}
\end{equation}
\end{say}

\begin{say}
Since $V$ is affine, we have an isomorphism $\pi_*
\sO_V(\bfmS+D)|_D (V) \cong \sO_V$. We fix such an isomorphism.
%So, again, as in the previous sections, $\pi_*   \sO_V(\bfmS +D)|_{D}  (V) =  \mathscr O_V$.
As in \S \ref{rtails},
consider the  short exact sequence:
$$0 \lra  \mathscr O_\cX(\bfmS+D) \lra  \mathscr O_\cX (\bfmS+S_j+ D)
\lra   \mathscr O_\cX(\bfmS +S_j+ D)|_S \lra 0.$$ Because
$\sO_\cX(\bfmS+S_j+D)|_{S_j}= \sO_\cX (S_j)^{m_j+1}|_{S_j}$,  we
have $\pi_* \sO_\cX(\bfmS+S_j+D)|_{S_j} = \sO_V(-(m_i+1)V_j)$,
hence we obtain a short  exact sequence of locally free sheaves
\begin{equation}\label{DDr-tails}
0 \lra \pi_* \sO_\cX (\bfmS+D) \lra \pi_*  \mathscr
O_\cX(\bfmS+S_j+D) \lra  \sO_V (-(m_j+1) V_j) \lra 0
\end{equation}
because $R^1 \pi_* \mathscr O_Z(mS+D)=0$.
Then by arguing that $\Ext^1=0$, we obtain that the sequence splits; then by induction, we get
\end{say}

\begin{lemm}
\label{lemmSplit-rtails}
 The sequence (\ref{DDr-tails}) splits and consequently
$$\pi_*  \sO_\cX (\bfmS+D) = \sO_V \oplus \bigoplus_{i=1}^r \bigoplus_{k=1}^{m_i} \sO_V (-k V_j). $$
\end{lemm}

\begin{say}
Let $$\sV_1 = \bigoplus_{i=1}^r \sO_V (-V_i).$$  Then we have a
canonical  map $\beta$,
 \begin{equation}\label{betarmoduli-rtails}
\begin{CD}
\sV_1  @>>> \sO_V \\
(h_1, \cdots, h_r)  @>>> \sum_i t_i h_i
\end{CD}
\end{equation}
which induces a canonical exact sequence
 \begin{equation}\label{CoreSequencermoduli-rtails}
\begin{CD}
0 \lra  \ker \beta  @>>> \sV_1  @>{\beta}>> \sO_V  @>{\gamma}>>
\sO_V |_W  \lra 0
\end{CD}
\end{equation}
where $W=\cap V_i$.
\end{say}

Set $$\sV_{{\bf m},0} =  \sO_V  \oplus \bigoplus_{ j=1}^r
\bigoplus_{k=2}^{m_j} \sO_V (-kV_j).$$ We have

\begin{prop}
\label{propModuli-rtails}
 Assume
%\footnote{When some $m_j =0$, we can ignore $W_j$, $S_j$, and the family $C_j$ from the consideration, this tail will not affect
%the sheaf $\pi_* \mathscr O_Z(\bfmS)$, and the problem is reduced to the previous $(r-1)$-tail case. This is analogous to
%the situation of \S \ref{1-tailr}.}
 that $m_j \ne 0$ for all $1 \le j \le r$.  Then up to isomorphism, we  have  a commutative diagram
\begin{equation}
\label{mianRtailmoduli}
\begin{CD}
0 \lra  \ker \beta  @>>> \sV_1  \\
@VVV @VVV \\ 0 \lra \pi_* \mathscr O_\cX (\bfmS) @>{\alpha_{\bf
m}}>> \pi_* \mathscr O_\cX (\bfmS+D)
\end{CD}
\end{equation}
\begin{equation*}
\begin{CD}
@>{\beta}>> \sO_V  @>{\gamma}>>  \sO_V |_W  \lra 0 \\
&&@VVV  @VVV  \\
@>{\beta_{\bf m}}>>  \sO_V @>{\gamma_{\bf m}}>>    R^1\pi_*
\mathscr O_\cX (\bfmS) \lra 0
\end{CD}
\end{equation*}
where the first two downward arrows are inclusions, and the last
two are isomorphisms.  Moreover, we have
\begin{enumerate}
\item $\pi_*  \mathscr O_\cX (\bfmS+D) = \sV_{{\bf m},0}  \oplus
\sV_1$. \item $\pi_* \mathscr O_\cX (\bfmS) \cong \ker \beta_{\bf
m} = \sV_{{\bf m},0}  \oplus \ker \beta $. \item $\beta_{\bf m}
|_{\sV_1} = \beta$.
\end{enumerate}
\end{prop}
\begin{proof}
It follows from the general proof in \cite{HL08}\footnote{It would
be nice to find an elementary proof of this proposition}.
\end{proof}

\subsection{The direct image sheaf around deeper strata:  general case}
\label{generalStrata}

\begin{say}
We let $\cX \lra V$ be a flat family of nodal curves of arithmetic
genus one. We let $0\in V$ be a fixed point with the fiber
$$\cX_0 = \bigcup_{i=0}^r C_i$$ such that
$C_0$ is a nodal elliptic curve, $C_i$ are connected trees of
rational curves attached to $C_0$ at distinct smooth points.
\end{say}

\begin{say}
 For
each $1 \le i \le r$, let $C_i^0$ be the  irreducible component of
$C_i$ that meets $C_0$. Note that as a nodal elliptic curve, $C_0$
contains either a smooth elliptic curve or a unique loop of
rational curves. This is the core $C_e$ of $\cX_0$.
\end{say}

\begin{say} We let $\sL$ be an invertible sheaf on $\cX$ whose
restriction to $C_0$ is trivial and whose restriction to $C_i^0$
are effective and with positive degree ($1 \le i \le r$). Thus,
$C_0$ is the largest subcurve of $\cX_0$ containing the core $C_e$
such that $\sL|_{C_0}$ is trivial.
\end{say}

\begin{say}
As in \S \ref{generalCase-r}, by shrinking $V$ if necessary, we
can choose sections $S_i$, $1 \le i \le r$, and $D$ such that
$\sL$ is of the form $\mathscr O_\cX (\bfmS)$ for some ${\bf m }
\ge {\bf 1}$.
\end{say}

\begin{say}
 For each $C_i$, $1 \le i \le r$,
let $$C_{i1}, \cdots, C_{in_i}$$ be the list of all irreducible
(ghost) rational curves through which $C_i$ is  attached to $C_e$.
Observe that some (ghost) rational curves   $C_{ih}$ are allowed
to repeat as $C_{jk}$ for some $j \ne i$.
\end{say}

\begin{say}
We have local variables  $t_i$, $t_{ij}$,  $1 \le i \le r$, $1 \le
j \le n_i$,
 each of which is the smoothing parameter of the  corresponding node.
Let $V_{i} = \{t_{i}=0\}$ and   $V_{ij} = \{t_{ij}=0\}$, $1 \le i
\le r$, $1 \le j \le n_i$.

Let $$\sV_1: = \sV_{\gamma,1} = \bigoplus_{i=1}^r \sO_V (-V_{i} -
\sum_{j=1}^{n_i} V_{{ij}}).$$
\end{say}

\begin{say}
The reason that  $\sO_V (-V_{i} - \sum_{j=1}^{n_i} V_{{ij}})$
occurs is because of the sequence
$$0 \lra  \mathscr O_\cX(\bfmS+D) \lra  \mathscr O_\cX(\bfmS+S_j+ D)
\lra   \mathscr O_\cX(\bfmS +S_j+ D)|_S \lra 0$$ and $\pi_*
\sO_Z(\bfmS+S_i+D)|_{S_i} = \sO_Z (S_i)|_{S_i}= \sO_V (-V_{i} -
\sum_{j=1}^{n_i} V_{{ij}})$.
\end{say}

\begin{say}
Then we have a canonical  map $\beta$,
 \begin{equation}\label{betaModuli}
\begin{CD}
\sV_1  @>>> \sO_V \\
(h_1, \cdots, h_r)  @>>> \sum_{i=1}^r t_i \prod_{j=1}^{n_i}
t_{{ij}} h_i
\end{CD}
\end{equation}
which induces a canonical exact sequence
 \begin{equation}\label{CoreSequencerModuli}
\begin{CD}
0 \lra  \ker \beta  @>>> \sV_1  @>{\beta}>> \sO_V  @>>> \sO_V |_W
\lra 0
\end{CD}
\end{equation}
where $W = \cap_{i=1}^r V_i$.
\end{say}

\begin{say}
Set $$\sV_{{\bf m},0} := \sV_{\gamma, {\bf m},0} = \sO_V  \oplus
\bigoplus_{ i=1}^r \bigoplus_{k=2}^{m_i} \sO_V (-k(V_{i}
+\sum_{j=1}^{n_i} V_{{ij}})).$$
\end{say}

\begin{prop}
\label{theoModuli}
 Assume
%\footnote{When some $m_j =0$, we can ignore $W_j$, $S_j$, and the family $C_j$ from the consideration, this tail will not affect
%the sheaf $\pi_* \mathscr O_Z(\bfmS)$, and the problem is reduced to the previous $(r-1)$-tail case. This is analogous to
%the situation of \S \ref{1-tailr}.}
 that $m_j \ne 0$ for all $1 \le j \le r$.  Then up to isomorphism, we  have  a commutative
 diagram
\begin{equation}
\label{Moduli}
\begin{CD}
0 \lra  \ker \beta  @>>> \sV_1  \\
@VVV @VVV \\ 0 \lra \pi_* \mathscr O_\cX (\bfmS) @>{\alpha_{\bf
m}}>> \pi_* \mathscr O_\cX (\bfmS+D)
\end{CD}
\end{equation}
\begin{equation*}
\begin{CD}
@>{\beta}>> \sO_V  @>>>  \sO_V |_W  \lra 0 \\
&&@VVV  @VVV  \\
 @>{\beta_{\bf m}}>>  \sO_V
@>>> R^1\pi_* \mathscr O_\cX (\bfmS)    \lra 0
\end{CD}
\end{equation*}
where the first two downward arrows are inclusions, and the last
two are isomorphisms.  Moreover, we have
\begin{enumerate}
\item $\pi_*  \mathscr O_\cX (\bfmS+D) = \sV_{{\bf m},0}  \oplus  \sV_1$.
\item $\pi_* \mathscr O_\cX (\bfmS) \cong \ker \beta_{\bf m} = \sV_{{\bf m},0}  \oplus \ker \beta $.
\item $\beta_{\bf m} |_{\sV_1} = \beta$.
\end{enumerate}
\end{prop}
\begin{proof}  Let $V' = V_{i} +\sum_{j=1}^{n_i} V_{{ij}}$. Then this reduces to the case of $r$-tails (Proposition
\ref{propModuli-rtails}).  See also \cite{HL08} for the general
proof.
\end{proof}

\section{Extensions of sections on the central fiber}

\subsection{Motivation}

\begin{say}
Consider the flat family of nodal elliptic curves  $\pi: Z \lra
B$. In this section, we look at sections in $H^0(Z_0,
\sO_{Z_0}(mS)$ over the central fiber and study their extensions
by deformation theory. This is useful because of the following.
Set theoretically, our deformation space is the union $$\bigcup_{b
\in B} H^0(Z_b, \sO_Z(mS)|_{Z_b})^{\oplus n}.$$ The deformation
space is singular at 0 if the core curve of $Z_0$ is ghost due to
the nonvanishing of the first cohomology. Naively, this is caused
by the existence of some sections  in $H^0(Z_0, \sO_{Z_0}(mS))$
that do not extend to neighboring fibers. Thus, it is instrumental
to determine, at least for some examples, exactly what sections
can be extended and what sections can not. (The singularities
occurring in the genus one case is relatively simple comparing to
the more general case. For more general deformation of sections of
invertible sheaves, one may consult the book $``$Deformations of
Algebraic Schemes$"$ by Edoardo Sernesi.)
\end{say}

\begin{say}
 We will deal exclusively
with the example of the 1-tail case as in \S \ref{1-tail}.
However, both the results and arguments work for arbitrary  flat
families of nodal elliptic curves.
\end{say}

\subsection{Sections on the central fiber}
\label{Sections-1-tail}

\begin{say}
Let $b: Z \lra \Ao \times E$ be the blowup of $\Ao \times E$ at
$(0, e_0)$, where   $e_0$ is a fixed point on $E$. Let $\pi_1: \Ao
\times E \lra {\mathbb A}^1$ be the projection to the first factor
and $\pi= \pi_1 \circ b$.
%and $\pi= \pi_1 \circ b: Z \lra {\mathbb
%P}^1 \times E \lra {\mathbb P}^1.$
\end{say}

\begin{say}
Choose a generic point $e \in E$. Let $S, D$ be the proper
transform of $\Ao \times e_0$ and $\Ao \times e$, respectively. We
will consider the direct image  sheaf $\mathscr L_m = \pi_*
\sO_Z(mS)$.
\end{say}

\begin{say}
We write the central fiber $Z_0$ of $\pi$ as
$$Z_0 = C \cup C_a$$
where $C$ is elliptic and $C_a$ is rational.  Choose homogeneous
coordinates $[u_0,u_1]$ of $C_a$ such that $C \cap C_a =[1,0]$ and
$S \cap C_a = [0,1]$. Then the space of sections of $\sO_Z(mS)$
restricted to $C_a$ has a basis
$$u_0^m, u_0^{m-1}u_1, \cdots, u_1^m.$$
Since $\sO_Z(mS)$ restricted to $C$ is trivial,  sections on $C$
are constant, hence these constants are uniquely determined by the
values of sections on $C_a$ at the point $C \cap C_a =[1,0]$. This
shows that $\sO_Z(mS)$ restricted to $Z_0$ has $m+1$ {dimensional
global} sections {spanned} by
$$u_0^m, u_0^{m-1}u_1, \cdots,
u_0u_1^{m-1}, u_1^m.$$
\end{say}

\begin{say}
{We first observe that $u_0^m$ extends to $Z$. Indeed, the
tautological inclusion $\sO_Z\to \sO_Z(mS)$ induces a section
$\bone\in \Gamma\bl Z,\sO_Z(mS)\br$ that vanishes along $S$ of
order $m$. Thus this section restricts $Z_0$ must be the $u_0^m$
mentioned.}
\end{say}

\begin{say}
When $m=1$, we claim that  $u_1$ is the section that does not
extend to the nearby fiber. For this, note that $u_1$  does not
vanish on $S \cap C_a=[0,1]$. Consider
$$\bigcup_{t \ne 0} \sO_{Z_t} (S_t)$$ where $S_t = S \cap Z_t$,
then  a section at a nearby fiber must vanish at the point $S_t$.
As $t$ specializes to 0, the limit section on the central fiber
must vanish at the point $S \cap Z_0 = S \cap C_a = [0,1]$.
{Combined, since $u_0$ extends, $au_0+bu_1$ extends if and only if
$b=0$.}
\end{say}

\subsection{The first order extensions} %of sections of $\sO_Z(mS)|_{Z_0}$}

\begin{say}
We need to investigate how sections
$$u_0^m, u_0^{m-1}u_1, \cdots, u_1^m
$$
in $H^0(Z_0,\sO_Z(mS))$ can be extended to $$H^0(kZ_0,\sO_Z(mS))
\quad  \hbox{for $k\geq 2$}.$$  To achieve this, we use Gieseker's
idea and instead look at
$$H^0(kC+(k-1)C_a,\sO_Z(mS)).$$
\end{say}
%To see this, we first look at the subscheme $2C\sub Z$. Since
%\col{$Z \setminus S$ is covered by only $V_1$ and $V_2$, $2C$ can
%be covered by two open subsets $2C \cap V_1$ and $2c \cap V_2$.}
%They can be written explicitly as .....

\begin{say}
We begin with studying  the sheaf $\sO_Z(mS)|_{2C+C_a}$ and its
global sections $H^0(2C+C_a,\sO_Z(mS))$. For this, we need to gain
knowledge on $H^0(2C,\sO_Z(mS))$, which will be determined by the
long exact sequence of cohomology of the exact sequence
\begin{equation}\lab{exact1}
0\lra \sO_Z(mS-C)|_C\lra \sO_Z(mS)|_{2C}\lra \sO_Z(mS)|_C\lra 0.
\end{equation}
Alternatively, this is just
$$
0 \lra I/I^2 \lra \sO_Z / I^2 \lra \sO_Z / I \lra 0
$$
where $I$ is the ideal sheaf of $C$. In particular,
$\sO_Z(mS-C)|_C = N^{\vee}_{C \backslash Z}$.  By adjunction
formula,
 $$N^{\vee}_{C \backslash Z}= K_Z |_C = \sO_Z (E)|_C = \sO_C (C \cap C_a)= \sO_C (q)$$
 where $q= C \cap C_a$. As a consequence $H^0( N^{\vee}_{C \backslash Z}) \cong
 \kk$  and $H^1( N^{\vee}_{C \backslash Z})=0$.
\end{say}

\begin{say}
From the exact sequence \eqref{exact1}, we have
\begin{equation*}
0\lra H^0( N^{\vee}_{C \backslash Z}) \lra H^0(\sO_Z(mS)|_{2C})
\lra H^0(\sO_Z(mS)|_C)\lra 0.
\end{equation*}
Via the inclusion $\sO_{2C}(-C)\sub \sO_{2C}$, $H^0( N^{\vee}_{C
\backslash Z})$ has the standard section $t$ which restricts to
zero in $\sO_{2C \cap C_a}$.
\end{say}

\begin{say}
Since $H^0(\sO_Z(mS)|_C)$ consists of constant sections,   by
lifting constant sections over $C$ to constant sections over $2C$,
we obtain a canonical split
$$H^0(\sO_Z(mS)|_{2C})= H^0(\sO_Z(mS)|_C) \oplus H^0( N^{\vee}_{C \backslash
Z}).$$  Thus, an arbitrary  element of $H^0(\sO_Z(mS)|_{2C})$ can
be (canonically) expressed as $a + b t$ for scalars $a$ and $b$.
\end{say}

\begin{say}
Then consider the exact sequence
\begin{equation}\begin{CD}\lab{exact2}
0@>>>\sO_Z(mS)|_{2C+C_a}@>>> \sO_Z(mS)|_{2C}\oplus
\sO_Z(mS)|_{C_a} \\
@>{(\phi_1,\phi_2)}>> \sO_{2C\cap C_a}@>>> 0.
\end{CD}
\end{equation}
First, we have
$$\sO_{2C\cap C_a}=\kk[u_1]/(u_1^2).
$$
We already know that
$$H^0(\sO_Z(mS)|_{C_a})=\text{span of}\ u_0^m, u_0^{m-1}u_1, \cdots,
u_1^m.
$$
Let $\phi_i$ be as indicated in the sequence \eqref{exact2}, then
sections of $$H^0(2C+C_a,\sO_Z(mS))$$ are pairs
$$(w_1,w_2)\in
H^0(\sO_Z(mS)|_{2C}) \times H^0(\sO_Z(mS)|_{C_a})
$$
such that
$$\phi_1(w_1)=\phi_2(w_2).
$$
We can write $w_1 = a + b t$ and $w_2 = \sum_i a_i u_0^{m-i}
u_i^j$,  by solving
$$\phi_1(w_1) = a  = \phi_2(w_2)= a_0 + a_1 u_1,$$ we see
that $a=a_0$ and $a_1=0$. Hence we conclude
$$H^0(2C+C_a,\sO_Z(mS))= \text{span of the extensions of }\ u_0^m, u_0^{m-2}u_2, \cdots,
u_1^m.
$$
\end{say}

\subsection{Further extensions}

We now let $W_n=(n+1)C+nC_a$, viewed as a subscheme of $Z$.

\begin{say}
In the previous subsection, we have already determined the space
$H^0(W_1,\sO_Z(mS))$. To find sections of $H^0(W_k,\sO_Z(mS))$, we
shall use induction based on the exact sequence
\begin{equation}\lab{exactk}
0\lra \sO_Z(mS-W_k)|_{Z_0}\lra \sO_Z(mS)|_{W_{k+1}}\lra
\sO_Z(mS)|_{W_k}\lra 0.
\end{equation}
Taking global sections, we obtain exact sequence
$$0\lra H^0(\sO_Z(mS-W_k)|_{Z_0})\lra H^0(\sO_Z(mS)|_{W_{k+1}})
\mapright{\alpha} H^0(\sO_Z(mS)|_{W_k}) $$ $$\lra
H^1(\sO_Z(mS-W_k)|_{Z_0})
$$
Because $W_k=kZ_0+C$ and $\sO_Z(Z_0)\cong\sO_Z$,
$$\sO_Z(mS-W_k)|_{Z_0}\cong \sO_Z(mS-C)|_{Z_0}.
$$
\end{say}

\begin{say}
{\sl Claim}: $H^1(\sO_Z(mS-W_k)|_{Z_0})=H^1(\sO_Z(mS-C)|_{Z_0}) =
0.$
\end{say}

\begin{say}
To prove the claim, we use an exact sequence similar to
\eqref{exact2}
\begin{equation*} %\lab{exactSimilar2}
\begin{CD}
0 @>>>\sO_Z(mS-C)|_{Z_0} @>>> \sO_Z(mS-C)|_{C}\oplus
\sO_Z(mS-C)|_{C_a} \\
 @>{(\phi_1,\phi_2)}>> \sO_{C\cap C_a}@>>> 0.
\end{CD}
\end{equation*}
Then taking cohomology, we have
\begin{equation*}
\begin{CD}
0 \lra H^0(\sO_Z(mS-C)|_{Z_0}) \lra H^0(\sO_Z(mS-C)|_{C}) \oplus
H^0(\sO_Z(mS-C)|_{C_a} )  \\
\mapright{(\phi_1,\phi_2)} H^0(\sO_{C\cap C_a})  \lra  H^1
(\sO_Z(mS-C)|_{Z_0}) \\
\lra H^1(\sO_Z(mS-C)|_{C}) \oplus H^1(\sO_Z(mS-C)|_{C_a}) \lra 0.
\end{CD}
\end{equation*}
$(\phi_1,\phi_2)$ is easily seen to be surjective. Also, we have
$$H^1(\sO_Z(mS-C)|_{C}) = H^1(\sO_C(q)) = 0,$$
$$H^1(\sO_Z(mS-C)|_{C_a}) = H^1(\sO_{C_a}(m-1)) = 0.$$ Thus the claim is
proved.
\end{say}

\begin{say}
From the claim, and by induction on $k$, all the sections in
$H^0(\sO_Z(mS)|_{W_k})$ ($k \ge 1$) extend to
$H^0(\sO_Z(mS)|_{W_{k+1}})$.
\end{say}

%Using $$H^1(\sO_Z(mS-C)|_{Z_0})=H^0(\sO_Z(mS-C)\dual|_{Z_0}\otimes_{\sO_{Z_0}} \omega_{Z_0})\dual, $$
%and that $\omega_{{Z_0}}$ is defined by the exact sequence
%$$0\lra \omega_{{Z_0}}\lra\omega_C(p)\oplus \omega_E(p)\lra \kk\lra 0, $$
%where $\omega_C(p)$ and $\omega_E(p)\to \kk$ are the residue homomorphism at $p$,

\subsection{The other sheaves}
\begin{say}
Let $D_0$ and $D_1$ be two sections of $Z/\Po$ passing through two
general points of $C$. We now look at the sheaf
$$\sM=\sO_Z(mS+D_0-D_1).
$$
\end{say}

\begin{say}
First observe that $H^0(\sO_Z(mS+D_0-D_1)|_{C_a})=
H^0(\sO_{C_a}(m))$ has the space of sections generated by $$u_0^m,
u_0^{m-1}u_1, \cdots, u_1^m.$$ Next, using the sequence
\begin{equation*}
\begin{CD}
0\lra \sO_Z(mS+D_0-D_1)|_{Z_0} \lra \\ \sO_Z(mS+D_0-D_1)|_{C}\oplus
\sO_Z(mS+D_0-D_1)|_{C_a}
 \mapright{(\phi_1,\phi_2)} \sO_{C\cap
C_a}\lra 0,
\end{CD}
\end{equation*}
and the fact that $$H^0(\sO_Z(mS+D_0-D_1)|_{C}) = H^0(\sO_C (q_0
-q_1)) =0$$ where $q_i = D_i \cap C$, $i=0,1$, we see that each
$u_0^iu_1^{m-i}$, $i>0$, extends by zero to $\sO_Z(mS+D_0-D_1)$.
Thus
$$H^0(\sO_Z(mS+D_0-D_1)|_{Z_0})= \text{span of the extensions of }\
u_0^{m-1}u_1, \cdots, u_1^m.
$$
\end{say}

\begin{say}
To proceed further, we need to determine
$$H^0(\sI_{C\sub 2C}(mS+D_0-D_1))\equiv
H^0(\sO_Z (mS+D_0-D_1)|_{2C}).
$$
For this, we look at
\begin{equation}
\begin{CD}
0\lra \sI_{C\sub 2C}(mS+D_0-D_1) \lra \sO_Z(mS +D_0-D_1)|_{2C}\\
 \lra \sO_Z(mS+D_0-D_1)|_C\lra 0.
\end{CD}
\end{equation}
Note that since $\sI_{C\sub 2C}(mS+D_0-D_1)$ is annihilated by
$\sI_{C\sub 2C}$, it is naturally a sheaf of $\sO_C$-modules. And
as such, $\sI_{C\sub 2C}(mS+D_0-D_1)\cong\sO_C(q+q_0 -q_1)$
(recall that $q= C \cap C_a$).  Taking the long exact sequence, we
then get
\begin{equation}
\begin{CD}
0\lra H^0(\sI_{C\sub 2C}(mS+D_0-D_1)) \lra H^0(\sO_Z(mS
+D_0-D_1)|_{2C}) \\
\lra H^0(\sO_Z(mS+D_0-D_1)|_C) .
\end{CD}
\end{equation}
Because $\sO_Z(mS+D_0-D_1)|_C = \sO_C (q_0-q_1)$, we conclude that
the last term is zero. Hence we obtain
$$H^0(\sO_Z(mS +D_0-D_1)|_{2C}) = H^0(\sI_{C\sub
2C}(mS+D_0-D_1))$$ $$=H^0(\sO_C(q+q_0 -q_1))=\kk
$$
where the last isomorphism holds by Riemann-Roch.
\end{say}

\begin{say}
We claim that the evaluation homomorphism
$$\phi_1: H^0(\sI_{C\sub 2C}(mS +D_0-D_1))\lra \sO_{2C\cap C_a}
$$
has image $\kk\cdot u_1\sub \kk[u_1]/(u_1^2)=\sO_{2C\cap C_a}$.
 Indeed, by
evaluating sections of $\sI_{C\sub 2C}(mS +D_0-D_1)$, we obtain an
exact sequence
$$0\lra \sI_{Z_0\cap 2C\sub 2C}(mS +D_0-D_1)\lra
\sI_{C\sub 2C}(mS +D_0-D_1)\lra \kk\cdot u_1\lra 0.
$$
Because as $\sO_C$-modules,
$$\sI_{Z_0\cap 2C\sub 2C}(mS +D_0-D_1)\cong \sO_C(q+q_0-q_1-q);
$$
thus its first cohomology group vanishes. This proves the claim.
\end{say}

\begin{say}
Now consider the sequence
\begin{equation*}
\begin{CD}
0 \lra H^0(\sO_Z(mS+D_0-D_1)|_{2C+C_a}) \lra \\
H^0(\sO_Z(mS+D_0-D_1)|_{2C}) \oplus H^0(\sO_Z(mS+D_0-D_1)|_{C_a} )
\mapright{(\phi_1,\phi_2)} H^0(\sO_{2C\cap C_a}).
\end{CD}
\end{equation*}
 By the above, we may choose a generator $s$ of
$H^0(\sO_Z(mS+D_0-D_1)|_{2C})$ so that its image under $\phi_1$ is
$u_1$. Then it is easy to see  that $\sum_{i\geq 0} a_i u_0^{m-i}
u_1^j$  extends to $H^0(\sO_Z(mS+D_0-D_1)|_{2C+C_a})$ if and only
if $a_0=0$. So,  we  find that
$$H^0(W_1, \sM)=\text{span of extensions of}\  u_0^{m-1}u_1, \cdots,
u_1^m.
$$ This is the same as $H^0(Z_0, \sM)$.
\end{say}

\begin{say}
Applying the parallel method as in the previous discussion, one
can find that
$$H^1(\sO_Z(mS+D_0-D_1-W_k)|_{Z_0})=H^1(\sO_Z(mS+D_0-D_1-C)|_{Z_0})
= 0$$ and then conclude that there are no obstructions to extend
all sections from  $H^0(W_1, \sM)$
 to $H^0(W_k,\sM)$.
\end{say}

\subsection{The trivialization of $\sO_Z(mS)$}
%The materials in this subsection are never used in this note or
%\cite{HL08}. Nevertheless, we included here for possible future
%references.

\begin{say}
 We fix an analytic coordinate $z$ of the elliptic curve $E$ over
$e_0\in U_0\sub E$ so $e_0=(z=0)$; we let $t$ be the standard
coordinate of $\Ao$.
\end{say}

\begin{say}
We let $E\sub Z$ be the exceptional divisor; we let $U\sub Z$ be
the analytic neighborhood $\pi_2\upmo(U_0)$ of $E\sub Z$. We let
$[u_0,u_1]$ be the homogeneous coordinate of $E$ so that over $U$:
$$u_0z=u_1t.
$$
We let $V_0=U\cap(u_1=1)$, $V_1=U\cap(u_0=1)$ and
$V_2=(C-e_0)\times \Ao$. Then we have
\begin{enumerate}
\item the transition from $V_0$ to $V_2$: $(u_0,z)\mapsto
(z,t)=(z,u_0z)$; \item the transition from $V_1$ to $V_2$:
$(u_1,t)\mapsto (z,t)=(u_1t,t)$; \item the transition from $V_0$
to $V_1$: $(u_0,z)\mapsto (u_1,t)=(1/u_0,u_0^2z)$.
\end{enumerate}
\end{say}

\begin{say}
This way, we cover $Z$ by the open subsets: $V_0$, $V_1$ and
$V_2$.
\end{say}

\begin{say}
Since $S\cap V_2=\emptyset$, $\sO_Z(mS)|_{V_2}$ has an obvious
trivialization. Also, we assume
 that $S\cap V_1=\emptyset$, hence $\sO_Z(mS)|_{V_1}$ also
an obvious trivialization. We assume that $S$  intersects  $E$ at
$([0,1],0) \in V_0$.
\end{say}

\begin{say}
We use the following local trivializations of $\sO_Z(mS)$:
$$(u_0,z, \eta_0)  \in \sO_Z(mS)|_{V_0} \cong
V_0 \times \CC,$$
$$(u_1, t, \eta_1 ) \in \sO_Z(mS)|_{V_1}
\cong V_1 \times \CC,$$
$$(z, t, \eta_2)  \in
\sO_Z(mS)|_{V_2} \cong V_2 \times \CC.$$ From $V_0$ to $V_1$, the
transition function is the same as the transition function for
$\sO_{\Po}(m)$, hence we have $$u_0^m \eta_1 = \eta_0 u_1^m.$$
Since we use 1 to trivialize both $V_1$ and $V_2$, the transition
 from $V_1$ and $V_2$ is just identity $$\eta_2 = \eta_1.$$
The two transitions patch all the local trivializations.
\end{say}

\end{document}